\documentclass[final]{siamart}

\usepackage[algo2e,ruled,vlined,noend,linesnumbered]{algorithm2e}
\SetKwComment{tcc}{\{}{\}}
\SetCommentSty{scriptsize}
\SetKwInput{Input}{Input} \SetKwInput{Output}{Return}
\SetKwIF{If}{ElseIf}{Else}{if}{}{else if}{else}{endif}
\SetKwFor{While}{while}{}{endw}
\SetKwFor{ForEach}{for each}{}{endfch}
\usepackage{graphicx}
\usepackage{booktabs}
\usepackage{amsfonts}
\usepackage{array}
\usepackage{subcaption}
\usepackage{bigdelim}
\usepackage{xfrac}
\usepackage{multirow}
\usepackage{enumitem}

\newsiamthm{claim}{Claim}
\newsiamremark{remark}{Remark}
\newsiamremark{expl}{Example}
\newsiamremark{conjecture}{Conjecture}
\crefname{hypothesis}{Hypothesis}{Hypotheses}

\usepackage{algpseudocode}
\usepackage{xspace}
\usepackage{siunitx}  
\newcommand{\TheTitle}{A Root-Node Based Algebraic Multigrid Method}
\newcommand{\TheAuthors}{T. A. Manteuffel, L. N. Olson, J. B. Schroder, B. S. Southworth}
\headers{\TheTitle}{\TheAuthors}
\title{{\TheTitle}\thanks{Submitted to the editors 6/30/2016. 
  This work performed under the auspices of the U.S. Department of Energy
  by Lawrence Livermore National Laboratory under Contract DE-AC52--07NA27344,
  LLNL-JRNL-695797\@. This research was conducted with Government support under and awarded by DoD, Air Force Office of Scientific Research, National Defense Science and Engineering Graduate (NDSEG) Fellowship, 32 CFR 168a.
  Portions of this work were sponsored by the Air Force Office of Scientific Research under
  grant FA9550--12--1--0478. This work was  performed under the auspices of the U.S. Department of Energy under
  grant numbers (SC) DE-FC02--03ER25574 and (NNSA) DE-NA0002376 and
  Lawrence Livermore National Laboratory under contract B600360.
  }}

\author{  Thomas~A.~Manteuffel
  \thanks{Department of Applied Mathematics,
          University of Colorado at Boulder
          (\email{tmanteuff@colorado.edu}).}
  \and
  Luke~N.~Olson
  \thanks{Department of Computer Science,
          University of Illinois at Urbana-Champaign, Urbana, IL 61801
          (\email{lukeo@illinois.edu},
           \url{http://lukeo.cs.illinois.edu}).}
  \and
  Jacob~B.~Schroder
  \thanks{Center for Applied Scientific Computing,
          Lawrence Livermore National Laboratory
          (\email{schroder2@llnl.gov},
           \url{http://people.llnl.gov/schroder2}).}
  \and
  Ben~S.~Southworth
  \thanks{Department of Applied Mathematics,
          University of Colorado at Boulder
          (\email{ben.southworth@colorado.edu}).}
}

\ifpdf\hypersetup{  pdftitle={\TheTitle},
  pdfauthor={\TheAuthors}
}
\fi

\usepackage{xspace}
\newcommand{\saamg}{SA~AMG\xspace}
\newcommand{\cfamg}{CF~AMG\xspace}
\newcommand{\rnamg}{RN~AMG\xspace}
\DeclareMathOperator*{\argmin}{argmin}
\DeclareMathOperator*{\sspan}{ran}
		\newcommand{\code}[1]{{\textnormal{\texttt{#1}}}}

\makeatletter
\let\oldnl\nl\newcommand{\nonl}{\renewcommand{\nl}{\let\nl\oldnl}}\makeatother

\begin{document}
\maketitle

\begin{abstract}

This paper provides a unified and detailed presentation of root-node style
algebraic multigrid (AMG).  Algebraic multigrid is a popular and
effective iterative method for solving large, sparse linear systems that arise from discretizing
partial differential equations.  However, while AMG is designed for symmetric
positive definite matrices (SPD), certain SPD problems, such as anisotropic diffusion,
are still not adequately addressed by existing methods. Non-SPD problems pose
an even greater challenge, and in practice AMG is often not considered as a solver
for such problems.

The focus of this paper is on so-called root-node
AMG, which can be viewed as a combination of classical and aggregation-based multigrid.
An algorithm for root-node is outlined and a filtering strategy is developed,
which is able to control the cost of using root-node AMG, particularly on
difficult problems. New theoretical motivation is provided for root-node and
energy-minimization as applied to symmetric as well non-symmetric systems.
Numerical results are then presented demonstrating the robust ability of
root-node to solve non-symmetric problems, systems-based problems,
and difficult SPD problems, including strongly anisotropic diffusion,
convection-diffusion, and upwind steady-state transport, in a
scalable manner. New, detailed estimates of the computational cost of
the setup and solve phase are given for each example, providing
additional support for root-node AMG over alternative methods.
\end{abstract}

\begin{keywords}
  multigrid, algebraic multigrid, root-node, energy minimization, interpolation smoothing, anisotropic diffusion
\end{keywords}

\begin{AMS} 65F10, 65M22, 65M55 \end{AMS}

\section{Introduction}\label{sec:introduction}

Algebraic multigrid (AMG) methods such as classical AMG (CF AMG\footnote{Classical AMG or so-called Ruge-St\"uben AMG use a splitting of the degrees-of-freedom (DOFs) into coarse \textit{C-points} and fine \textit{F-points}, leading to the abbreviation CF.})~\cite{BrMcRu1984, RuStu1987} and
smoothed aggregation (SA AMG)~\cite{VaMaBr1996} are efficient solution techniques
for large, sparse linear systems. Algebraic multigrid was developed
specifically for symmetric positive-definite (SPD) systems that arise from
the discretization of elliptic partial differential equations (PDEs),
and software packages such as BoomerAMG~\cite{BoomerAMG}
in the hypre library~\cite{hypre}
demonstrate its parallel scalability to hundreds of thousands of cores~\cite{Baker2012}.

AMG targets solving a sparse linear system (typically SPD)
\begin{equation}\label{eqn:linear-system}
 A \mathbf{x} = \mathbf{b}
\end{equation}
with $O(n)$ work, where $\mathbf{x},\, \mathbf{b} \in \mathbb{R}^n$ and $A \in \mathbb{R}^{n\times
n}$.  This optimality is achieved through two complementary parts of a multigrid
method, relaxation and coarse-grid correction, which together
uniformly damp error of all frequencies (see Section~\ref{sec:background}).

However, there are certain SPD systems and many non-symmetric systems
for which AMG continues to struggle.
Problems with strongly anisotropic components, and problems arising in particle
transport, advective flow calculations, and strongly varying material
properties, among others, challenge the standard approaches to AMG, thus highlighting the need
for more robust methods. There have been a number of efforts in recent years to improve the convergence
and scope of applicability of AMG\@.  Adaptive methods focus on improving the multigrid hierarchy through
trial cycles in the setup phase~\cite{Brezina:2006bt,DAmbra:2013iwa,BrFaMaMaMcRu2005,BrMaMcRuSa2010}. Other methods
focus on modified or improved strength measures when forming coarse
grids~\cite{OlScTu2009,DAmbra:2013iwa,BrBrKaLi2015,Notay:2010um,Brandt:2011tu,Brannick:2012bl,BrBrMaMaMcRu2006,Brandt:2000vn,Brannick:2010hz,Livne:2004vt}.
Furthermore, generalizing interpolation through energy minimization~\cite{OlScTu2011} and other
methods~\cite{Gee:2009dy,Chan:2000tl,DeFaNoYa_2008,Wiesner:2014cy}
is used to improve the accuracy of interpolation between grid levels.
Nevertheless, many problems remain difficult for AMG to solve, while many `robust'
AMG methods suffer from high computational cost.

An AMG solver consists of a hierarchy of matrices, $\{A_\ell\}$, with the initial matrix
on level $\ell=0$, $A_0 := A$, and progressively smaller matrices on levels $\ell=1,2,\dots,L$.
Interpolation and restriction operators, also known as transfer operators, transfer
vectors between different levels of the hierarchy. For a given matrix, $A_\ell$, the next `coarser' matrix in the hierarchy (level $\ell+1$)
is generally developed in one of two different ways: using a CF-splitting
of points (\cfamg) or using an aggregation of points (\saamg).
A CF-splitting splits the set of all DOFs of matrix $A_\ell$
into a coarse set of C-points and a fine set of F-points, with C-points corresponding to DOFs on the coarse grid.  Transfer operators are then
defined using the CF-splitting, where values at C-points are restricted and interpolated
by injection and values at F-points use a linear combination of connected
neighboring points. In \saamg, a measure of the strength-of-connection (SOC) between nodes
is used to form `aggregates', which are disjoint sets of strongly connected nodes, where each
aggregate represents one node on the coarse grid, and transfer operators are
formed based on aggregates. A more detailed look at
SA AMG and CF AMG is given in Section~\ref{sec:background}.

Root-node AMG (\rnamg) uses a hybrid approach, wherein SA-type
strength-of-connection and aggregation routines are used to form aggregates. In
each aggregate, one node is chosen to be the `root-node', which corresponds to a
C-point, and other nodes are designated as F-points. Transfer operators are
then formed based on this CF-splitting together with aggregation.

Root-node AMG was initially identified in~\cite{OlScTu2011} as a small part of a
general, energy-minimizing framework to form interpolation operators in AMG\@.
The work in~\cite{Sc2012} implemented the RN
approach and demonstrated its potential as an effective and scalable solver for
strongly anisotropic, non-grid-aligned diffusion operators~---~problems which have
proven difficult for other multigrid methods. However, certain
problems required a large computational cost, especially in the
setup of the method.

This paper provides, for the first time, a unified and detailed presentation of \rnamg\
and how it combines many of the benefits of \cfamg\ and \saamg.  Root-node AMG
allows for classical point-wise decisions in the
setup, to help control complexity\footnote{``Complexity'' or ``cost'' refer to the overall computational cost of the method in terms of floating point operations.}
and provide theoretical motivation, as well as aggregation-style
construction, to facilitate the design of a multigrid solver based on the
spectral behavior of the problem.
New theoretical motivation is given for the pairing
of \rnamg\ with energy minimization, in both the symmetric and non-symmetric cases,
proving the equivalence of the energy-minimization process to
minimizing the difference with an optimal form of interpolation.
Moreover, a new interpolation filtering strategy is developed to limit the complexity
of the method, which proves critical for problems that require
large sparsity patterns in transfer operators.
Last, a numerical survey is presented to highlight the robustness and flexibility
of \rnamg\ in comparison to \cfamg\ and \saamg, including scalable convergence
for strongly anisotropic diffusion problems and a discontinuous, upwind discretization
of the steady-state transport equation. For each test problem, a detailed measure of computational cost or
complexity is provided, a novel addition to AMG literature, which
provides a complete picture of a method when coupled with convergence
factors.

In Section~\ref{sec:background}, background information on algebraic multigrid
methods is discussed, including current limitations and the basic motivation for
a RN-type algorithm.
The \rnamg\ method and algorithm are presented in
Section~\ref{sec:root-node}, along with a discussion of computational complexity
and a filtering process proposed to address cases of high complexity.
Section~\ref{sec:theory} provides theoretical
motivation for \rnamg. Numerical results are
provided in Section~\ref{sec:numerical}, and conclusions and future work
discussed in Section~\ref{sec:conclusions}.

\section{Background}\label{sec:background}

Multigrid methods, such as \saamg\ and \cfamg, involve two phases: (i) the setup
phase, where a multilevel solver hierarchy is constructed, and (ii) the solve
phase, where the constructed solver hierarchy is applied to solve the linear
system~\eqref{eqn:linear-system} to a desired tolerance.  Smoothed-aggregation AMG and \cfamg\ are distinguished by the setup
phase; that is, once a hierarchy is constructed, both methods execute the solve
phase in the same fashion.

In the following, a multigrid hierarchy consists of a set of matrices, starting with an
initial, fine matrix $A_0 \equiv A \in \mathbb{R}^{n\times n}$. Matrices for additional levels
in the hierarchy, $A_\ell \in \mathbb{R}^{n_\ell \times n_\ell}$, are then constructed based on
interpolation and restriction operators, $P_\ell\,:\,\mathbb{R}^{n_{\ell+1}}\rightarrow\mathbb{R}^{n_\ell}$
and $R_\ell\,:\,\mathbb{R}^{n_{\ell}}\rightarrow\mathbb{R}^{n_\ell+1}$, respectively,  via
$A_{\ell+1} = R_\ell A_\ell P_\ell$, where $n_{\ell+1} < n_\ell$. In the case of an SPD
matrix, restriction $R = P^T$.

The AMG solve phase iterates using two complementary parts: relaxation~---~e.g.,
weighted Jacobi~---~to reduce the high-energy error that is associated with large
eigenvalues in the operator, and coarse-grid correction, which targets
algebraically smooth error, $A_0 \mathbf{e} \approx 0$, associated with the small eigenvalues.
A two-grid solve proceeds as follows.  Pre-relaxation on $A_0 \mathbf{x}_0 = \mathbf{b}_0$ is applied,
and the resulting residual is then restricted to the coarse grid, $\mathbf{b}_1 = R_0 (\mathbf{b}_0 - A_0 \mathbf{x}_0)$,
which serves as the right-hand side for the coarse-grid equation,  $A_1 \mathbf{x}_1 = \mathbf{b}_1$,
where $A_1 = R_0 A_0 P_0$. The solution, $\mathbf{x}_1$, provides a coarse-grid
error correction that is interpolated back to the fine grid, $\mathbf{x}_0 \leftarrow \mathbf{x}_0 + P_0 \mathbf{x}_1$.
Last, post-relaxation is applied to the updated $\mathbf{x}_0$. Together, these three steps form
a two-level multigrid V{($\nu_{\textnormal{pre}}$, $\nu_{\textnormal{post}}$)}-cycle,
where $\nu_{\textnormal{pre}}$ refers to the number of pre-relaxations and
$\nu_{\textnormal{post}}$ to the number of post-relaxations. A full solve then
consists of using successive V-cycles to iterate on a vector until the relative residual norm,
$\|\mathbf{r}_0\|/\|\mathbf{b}_0\|$ is less than some tolerance~---~e.g., $10^{-8}$.

The effectiveness of this complementary process is explained by considering the corresponding two-grid error propagation
operator.  Let $\mathbf{e}^{(0)}$ be the initial error in approximating the solution to~\eqref{eqn:linear-system} and let $\mathbf{e}^{(0)} \leftarrow G \mathbf{e}^{(0)}$ represent
the error propagator for the relaxation method~---~e.g.,
$G=I-\omega D^{-1} A$ in the case of weighted Jacobi. Then (dropping
subscripts), the error after a two-grid cycle is given by
\begin{equation}\label{eq:twogrid}
  \mathbf{e}^{(1)} \leftarrow G^{\nu_2} \left(I - P {(R A P)}^{-1} R A \right) G^{\nu_1} \mathbf{e}^{(0)},
\end{equation}
where $P {(R A P)}^{-1} R A$ is a projection onto $\mathcal{R}(P)$. In the case of SPD $A$,
$R = P^T$, and this is an $A$-orthogonal projection. In either case, if $Ge^{(0)} \in \mathcal{R}(P)$,
then the iteration is exact. In other words, if interpolation is complementary and accurate for modes
not effectively reduced by relaxation, then error reduction with~\eqref{eq:twogrid} will be large.

Algebraic multigrid methods attempt to automatically determine interpolation and coarse-grid
operators ($P_\ell$ and $A_\ell$) that yield optimal error reduction
with~\eqref{eq:twogrid}.  The two standard approaches are \saamg\ and \cfamg, which
are outlined in Sections~\ref{sec:background:sa}~and~\ref{sec:background:cf}, respectively.
Broadly, \saamg\ defines a coarse problem through
an \textit{aggregation} of nodes (see Figure~\ref{fig:graph-aggregates}), while
\cfamg\ defines a coarse problem through a splitting of the DOFs
into coarse C-points and fine F-points (see
Figure~\ref{fig:graph-splitting}).  Each offer advantages as noted in the
following descriptions.
\begin{figure}
  \begin{center}
  \begin{subfigure}[b]{0.48\textwidth}
    \includegraphics[width=\textwidth]{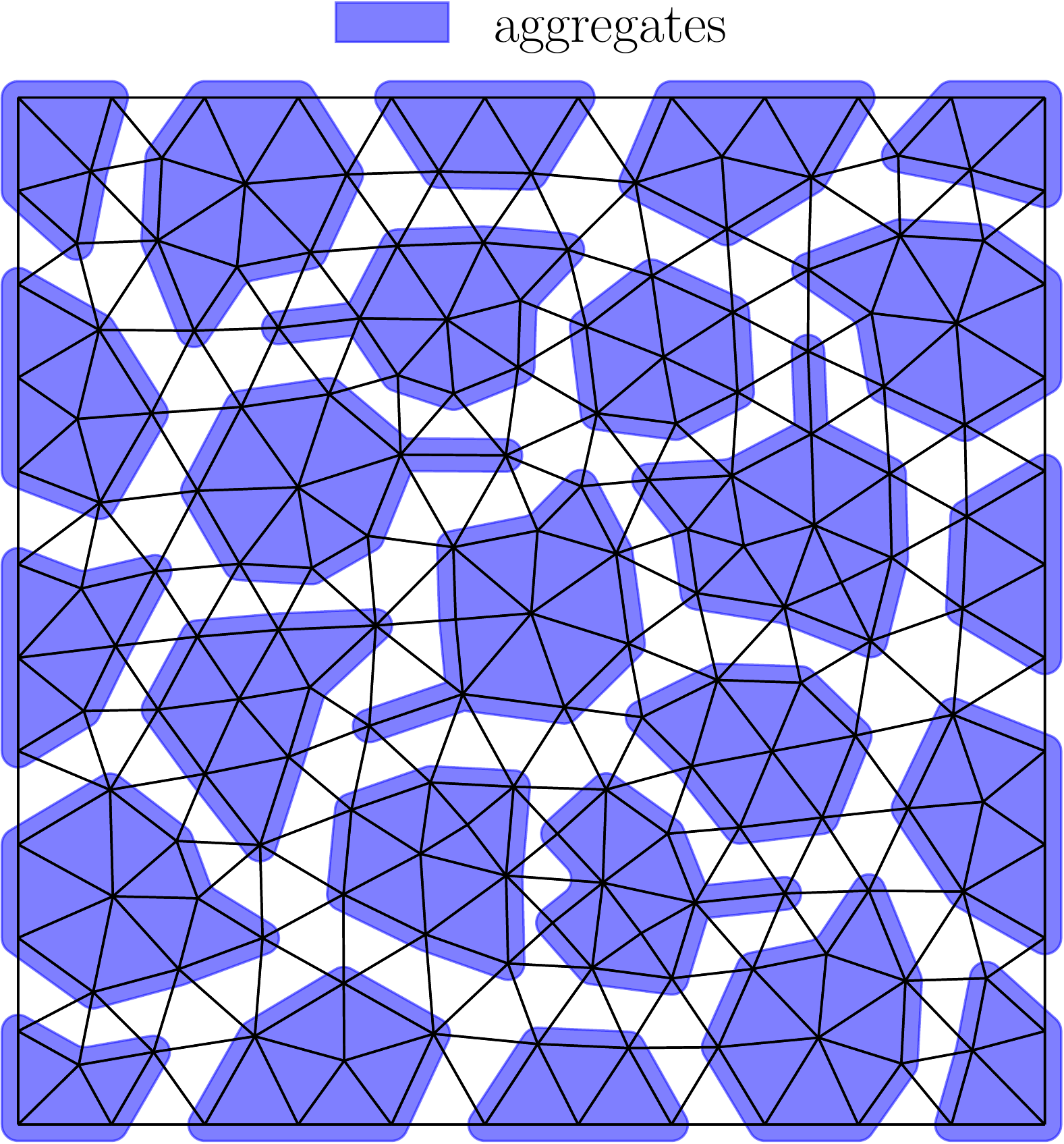}
    \caption{Aggregation}\label{fig:graph-aggregates}
  \end{subfigure}
  \hfill
  \begin{subfigure}[b]{0.48\textwidth}
    \includegraphics[width=\textwidth]{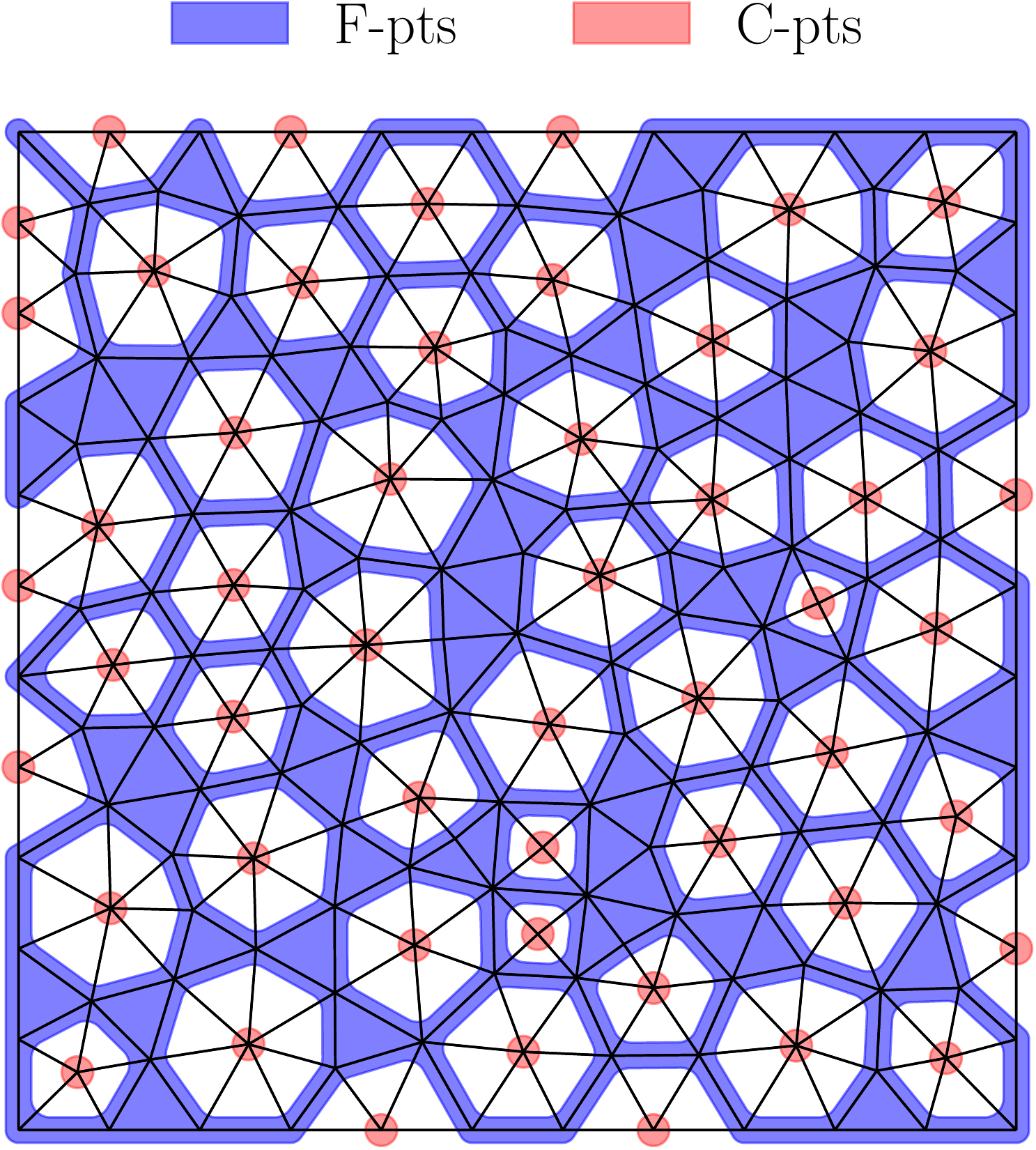}
    \caption{$C/F$ splitting}\label{fig:graph-splitting}
  \end{subfigure}
  \caption{Example \saamg\ and \cfamg\ coarsening for a linear finite element approximation to a Laplace operator.  The fine level problem has 191 DOFs. In this example, aggregation yields 25 coarse DOFs (aggregates), while a CF-splitting yields 51 coarse DOFs (C-points).}\label{fig:graph-example}
  \end{center}
\end{figure}

\subsection{AMG based on smoothed aggregation}\label{sec:background:sa}

The effectiveness of \saamg\ relies on \textit{a priori} knowledge of
algebraically smooth error in the form of candidate vectors, $B$.
These vectors generally represent the lowest energy modes of the governing PDE with
no boundary conditions~---~e.g., the constant for diffusion and rigid-body-modes
for elasticity~\cite{VaMaBr1996}.  Using $B$ (possibly determined \textit{a priori}), the
setup phase first uses a strength measure on the connectivity of nodes to define a SOC matrix, $S$, which
is used to
identify so-called aggregates (see Figure~\ref{fig:graph-aggregates}).
The goal of the strength measure is to ensure that algebraically smooth error
at each DOF in an aggregate strongly correlates with algebraically smooth
error at other DOFs in that aggregate.  This holds for
model problems, where algebraically smooth error
over each aggregate is well represented by the
restriction of the
candidate vectors
to the aggregate.  Consequently, this \emph{injection} of the candidate vectors
over each aggregate is what leads to the initial representation of interpolation,
termed the tentative interpolation operator $T$, with candidate vectors exactly
in the range of $T$. Each aggregate
corresponds to one block in the block-diagonal operator $T$, with one column
for each candidate and one row for each point in the aggregate.

As an example, consider the 1D-Laplace operator on an eight-node mesh, using
standard finite differences, with candidate vectors $B=[{\bf 1}, {\bf x}]$.
This yields three aggregates,
$\mathcal{A}_i, i=0,1,2$, as shown in Figure~\ref{fig:agg1d-setup}.  The
aggregation pattern matrix, $C$, coupled with the candidate vectors, $B$, yield the following
\begin{equation}\label{eq:exampleinterp}
  C = \left[\begin{array}{ccc}
      * &  & \\
      * &  & \\
        & * &  \\
        & * &  \\
        & * &  \\
        &   & *\\
        &   & *\\
        &   & *\\
    \end{array}\right],
  \hspace{1ex}  B = \left[\begin{array}{cc}
      1 & 1/9\\
      1 & 2/9\\
      1 & 3/9\\
      1 & 4/9\\
      1 & 5/9\\
      1 & 6/9\\
      1 & 7/9\\
      1 & 8/9\\
  \end{array}\right]
  \hspace{1ex}\rightarrow\hspace{1ex}  T = \left[
        \begin{array}{r r r r r r}
          1             & -1 &              &    &              & \\
          1             & 1  &              &    &              & \\
                        &    & 1 & -1 &              & \\
                        &    & 1            & 0  &              & \\
                        &    & 1            & 1  &              & \\
                        &    &              &    & 1 & -1\\
                        &    &              &    & 1            & 0\\
                        &    &              &    & 1            & 1\\
        \end{array}
      \right]
      D,
\end{equation}
where $D=\textrm{diag}{([2, 2, 3, 2, 3, 2])}^{-1/2}$ is a diagonal matrix to
normalize each column in the $l^2$-norm.  Each block of $T$ in~\eqref{eq:exampleinterp} is an orthogonal basis for the restriction of $B$ to each aggregate.

In the case of a single candidate vector, $B=[{\bf 1}]$, $T$ consists of columns
$0$, $2$, and $4$ in~\eqref{eq:exampleinterp}.  These columns are plotted
in Figure~\ref{fig:agg1d-1-CSR-SA} (dashed lines).  It is important to note that
each column is nonzero only on its respective aggregate.  To
improve the accuracy of interpolation for algebraically smooth modes~---~i.e.,
to make $\mathcal{R}(P)$ more complementary to relaxation~---~the
columns of the tentative interpolation operator are smoothed~---~e.g., with
weighted Jacobi~---~to form $P$ (Figure~\ref{fig:agg1d-1-CSR-SA}, solid lines).  As the range of $P$ becomes richer,
so does the nonzero footprint in the operator. Indeed, nonzero elements
of the middle column of $P$ in Figure~\ref{fig:agg1d-1-CSR-SA} (solid green)
now overlap into the neighboring aggregates.

\begin{figure}[hb!]
  \centering
  \begin{subfigure}[b]{0.48\textwidth}
    \includegraphics[width=\textwidth]{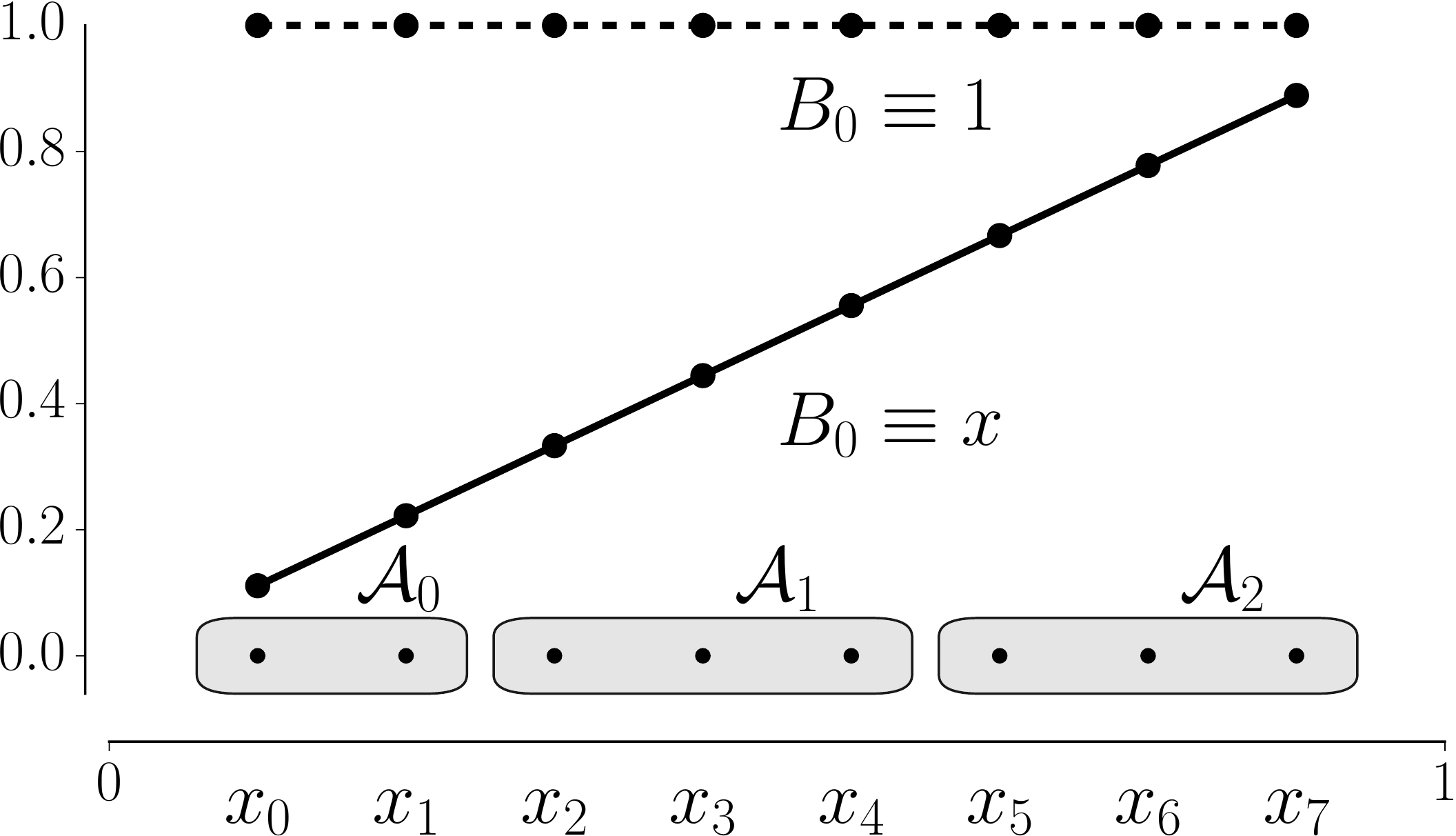}
    \caption{Three aggregates and two candidate vectors.}\label{fig:agg1d-setup}
  \end{subfigure}
  \hfill
  \begin{subfigure}[b]{0.48\textwidth}
    \includegraphics[width=\textwidth]{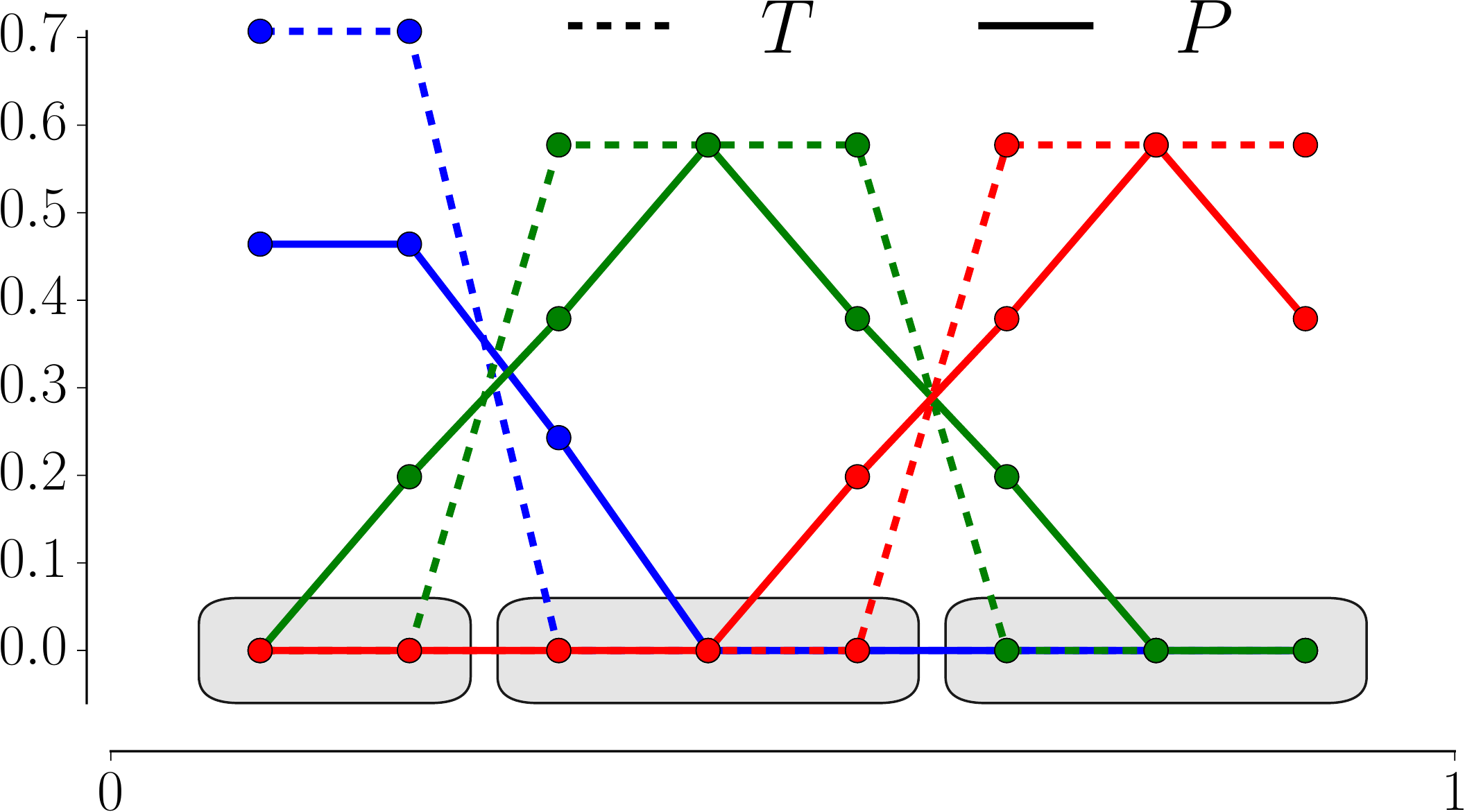}
    \caption{Columns of interpolation using a single candidate with SA.}\label{fig:agg1d-1-CSR-SA}
  \end{subfigure}
  \caption{1-D Laplace example using \saamg.}\label{fig:agg1d-SA}
\end{figure}

Given a set of candidate vectors $B$ on the current grid that are exactly represented in the range of $T$,
coarse-grid candidate vectors are constructed as the pre-image of $B$ under $T$. The
motivation is that the pre-image of low energy vectors for $A_{\ell}$ should be low-energy
vectors for $A_{\ell+1}$.
Let the orthogonal projection onto $\mathcal{R}(T)$ be given by $\pi_T = T (T^T T)^{-1} T^T$. Since the columns of $T$ are orthonormal, $\pi_T = T T^T$, and requiring $\pi_T B = B$, that
is $B\in\mathcal{R}(T)$, results in a coarse-grid pre-image of $B$ under $T$ given by
$B_c = T^T B$. Since $B$ is assumed to be a low energy mode, forming $P$ by applying
smoothing iterations to $T$ to improve multilevel convergence~\cite{Van:2001bw} keeps
$B$ close to the range of $P$.

A general \saamg\ setup algorithm is given in Algorithm~\ref{alg:SAAMG}.

\begin{algorithm2e}[!ht]
  \caption{\texttt{SA\_setup{()}}}\label{alg:SAAMG}
   \DontPrintSemicolon   \KwIn{      \begin{tabular}[t]{l l}
      $A_0$:               & fine-grid operator\\
      $B_0$:               & fine-grid candidate vectors\\
      \code{max\_size}:    & threshold for max size of coarsest problem \\
      \end{tabular}
   }
   \nonl\;
   \KwOut{      \begin{tabular}[t]{l}
      $A_{1}, \dots, A_{L}$,\\
      $P_{0}, \dots, P_{L-1}$
      \end{tabular}
   }
   \nonl\;
   $\ell = 0$\;
   \While{$\code{size}(A_\ell) > \code{max\_size}$}{      $S_{\ell}$           = \code{strength}{($A_{\ell}$)}\tcc*[r]{Strength-of-connection}
      $\mathcal{A}_{\ell}$ = \code{aggregate}{($S_{\ell}$)}\tcc*[r]{Aggregation}
      $T_{\ell}$, $B_{\ell+1}$ = \code{inject}{($\mathcal{A}_{\ell}$, $B_{\ell}$)}\tcc*[r]{Form tentative interpolation and coarse candidates}
      $P_{\ell}$          = \code{smooth}{($A_{\ell}$, $T_{\ell}$)}   \tcc*[r]{Smooth $T_{\ell}$}
      $A_{\ell+1} = P_{\ell}^T A_{\ell} P_{\ell}$\tcc*[r]{Coarse-grid operator}
      $\ell = \ell + 1$\;
    }
\end{algorithm2e}

\subsection{AMG based on coarse-fine splittings}\label{sec:background:cf}

In contrast to \saamg, CF AMG\ builds a multilevel hierarchy through a
CF-splitting. On each level $\ell=0,1,\dots,L$, the index set of DOFs, $\Omega^{\ell} = \{0,\dots,n_{\ell}\}$, is
split into $\Omega^{\ell} = \mathcal{C}_{\ell} \cup \mathcal{F}_{\ell}$, where $\mathcal{C}_{\ell} \cap \mathcal{F}_{\ell} = \emptyset$.  The
set $\mathcal{C}_\ell$ defines the coarse-level DOFs so that $n_{\ell+1} = |\mathcal{C}_\ell|$.
Similar to \saamg, a strength measure, $S$, is used to determine the splitting
so that algebraically smooth error at F-points are well-approximated
by evaluating neighboring C-points.  Then interpolation,
$P\,:\,\mathbb{R}^{n_{\ell+1}}\rightarrow\mathbb{R}^{n_{\ell}}$, is formed as
\begin{equation}\label{eq:cfinterp}
  P =
  \begin{array}{c@{}c}
    \left[
      \begin{array}{c}
        W \\
        I\\
      \end{array}
    \right]
  &
  \begin{array}{l}
    \} \,\,\textnormal{F-points}\\
    \} \,\,\textnormal{C-points}\\
  \end{array}
  \end{array},
\end{equation}
where $W\in\mathbb{R}^{|\mathcal{F}_\ell|}\times\mathbb{R}^{|\mathcal{C}_\ell|}$
is a sparse matrix, with entries chosen to approximate smooth error at
F-points (rows of $W$) as a linear combination of C-points (columns of $W$).

The form of interpolation in~\eqref{eq:cfinterp} highlights two attributes of
interpolation in \cfamg\ that distinguishes it from \saamg.  The first is that the source of
complexity and accuracy is clear: the number of nonzeros in $W$ controls both
the accuracy and cost of interpolation.  This is explored further in
Section~\ref{sec:cost_overview}.  Second, coarse values are \textit{injected} to the finer
grid through the identity in the bottom block.  This in turn
ensures linear independence of the columns of $P$, an important feature
not guaranteed in \saamg.

A general \cfamg\ setup algorithm is given in Algorithm~\ref{alg:cfamg}.
\begin{algorithm2e}[!ht]
  \caption{\texttt{CF\_setup{()}}}\label{alg:cfamg}
   \DontPrintSemicolon   \KwIn{      \begin{tabular}[t]{l l}
      $A_0$:               & fine-grid operator\\
      \code{max\_size}:    & threshold for max size of coarsest problem \\
      \end{tabular}
   }
   \nonl\;
   \KwOut{      \begin{tabular}[t]{l}
      $A_{1}, \dots, A_{L}$,\\
      $P_{0}, \dots, P_{L-1}$
      \end{tabular}
   }
   \nonl\;
   $\ell = 0$\;
   \While{$\code{size}(A_\ell) > \code{max\_size}$}{            $S_{\ell}$           = \code{strength}{($A_{\ell}$)}\tcc*[r]{Strength-of-connection}
      $\mathcal{C}_{\ell}$,
      $\mathcal{F}_{\ell}$ = \code{splitting}{($S_{\ell}$)}\tcc*[r]{$C/F$-splitting}
      $W$ = \code{weights}{($S_{\ell}$, $A_{\ell}$, $\mathcal{C}_{\ell}$, $\mathcal{F}_{\ell}$)}\tcc*[r]{Interpolation weights}\label{line:weights}
      $P_{\ell}$          = $\left[\begin{array}{c}W \\ I\end{array}\right]$\tcc*[r]{Form interpolation}
      $A_{\ell+1} = P_{\ell}^T A_{\ell} P_{\ell}$\tcc*[r]{Coarse-grid operator}
      $\ell\ = \ell\ + 1$\;
    }
\end{algorithm2e}

\subsection{Benefits and limitations}\label{sec:limitations}

In \cfamg, it is assumed that the constant
vector is representative of the near null space of the underlying problem,
as in the case of a Poisson problem.  As a result, convergence can degrade
if this is not an accurate assumption.
Moreover, the interpolation formulas for \cfamg\ are static and offer
no immediate ways to improve them for more difficult problems.

One benefit of \cfamg\ is that it provides a structure for controlling
sparsity through the determination of weights (see line~\ref{line:weights} in
Algorithm~\ref{alg:cfamg}). Each row of $W$ (see~\eqref{eq:cfinterp})
represents an interpolation formula for a given F-point from surrounding
C-points. As a result, this leads
to simple filtering strategies in $P$~\cite{DeFaNoYa_2008} that
eliminate small entries in each row. In \saamg, interpolation is based on
having algebraically smooth columns of $P$.  Consequently, existing methods to
explicitly filter entries in $P$ are limited, as removing entries can greatly reduce smoothness~\cite{Chan:2000tl}.
The structure of interpolation operators in \cfamg\ also allows for theoretical
results that are not feasible for arbitrary transfer operators as in \saamg~\cite{FaVa2004, Va2008, Notay:2014uc}, and which are used in Section~\ref{sec:theory}
to motivate \rnamg.

For problems in which a single global vector adequately represents the algebraically
smooth error, \saamg\ and \cfamg\ can each be effective.
One benefit of \saamg\ is that the method allows for multiple
candidate vectors to help define the range of interpolation in order
to improve the coarse-grid correction. However, the computational cost of iterating tends
to increase substantially when additional candidate vectors are included. Since each candidate
vector occupies a column in $P$ for each aggregate (see~\eqref{eq:exampleinterp}), additional
candidate vectors quickly increase the number of DOFs and nonzeros in coarse-level operators.

As an example, consider a two-level multigrid method for a 2D-Laplacian discretized with linear, quadrilateral
finite elements over a $50\times 50$ uniform grid.  In this case, $A_0
\in\mathbb{R}^{2500\times 2500}$, with 21904 nonzero elements. Using the symmetric
strength matrix~---~i.e. $S_{ij} =1$ if $A_{ij}/\sqrt{A_{ii}A_{jj}}>0.25$~\cite{VaMaBr1996}~---~along
with standard, greedy aggregation~\cite{VaMaBr1996} yields 289 aggregates.  This results in
$P_0\in \mathbb{R}^{2500\times 298}$ in the case of a single candidate vector, and
a coarse-grid operator $A_1\in\mathbb{R}^{289\times 289}$ with 2401 nonzeros. Using
two candidate vectors, $P\in \mathbb{R}^{2500\times
578}$, and the coarse-grid operator, $A_1\in\mathbb{R}^{578\times 578}$, has
9604 nonzeros.  In this two-grid example, there is approximately a 30\% increase in the
total number of matrix nonzeros in the hierarchy, which will correspond to a comparable
increase in the cost of each iteration\footnote{See~\cite{KeMaSc2015} for an example
where a scalar diffusion-like problem requires multiple candidate vectors.}. A key feature of the \rnamg\ method
introduced in Section~\ref{sec:root-node} is that the growth in complexity is mitigated when
incorporating multiple candidate vectors in the range of $P$.

Additionally, \saamg\ provides a simple way to improve interpolation operators
through the interpolation smoothing process.  While classical \saamg\ only uses one
weighted-Jacobi iteration to improve $P$, multiple iterations as well as
smoothing with a \textit{filtered} operator to further improve $P$ have been used~\cite{Chan:2000tl,MaBrVa1999}.
Nevertheless, because each traditional smoothing iteration expands the sparsity
pattern of $P$, this process is limited, a problem which is overcome in \rnamg\ through
\textit{a priori} fixed sparsity patterns of transfer operators.

\subsection{Computational cost}\label{sec:cost_overview}

The computational kernel in the multigrid setup and solve phases is a
sparse matrix-vector product (SpMV).  Thus, a representative measure of the cost
of an AMG solver is the number of floating point operations relative to one SpMV with
the initial matrix. This measure is referred to as a \textit{work unit} (WU), where one WU
is the cost of computing a SpMV on the finest level.
Both \cfamg\ and \saamg\ often yield minimal setup costs or setup complexity (SC),
but as more features are introduced~---~e.g., improved SOC methods and energy minimization~---~the SC may grow. In contrast to the fixed cost of setup, the solve cost
depends on the number of iterations or cycles taken, which in turn depends on the stopping residual tolerance.
Consequently, the \textit{cycle} complexity (CC), denoted $\chi_{\textnormal{CC}}$, is defined as the number of
WUs required for each multigrid cycle, and is used to measure the solve cost.
A similar measure is the \textit{operator} complexity (OC), denoted $\chi_{\textnormal{OC}}$, which models
the cost of a multigrid hierarchy as the ratio of the total number of nonzeros on all levels
to the number of nonzeros on the finest-level:
\begin{equation}\label{eq:oc}
  \chi_{\textnormal{OC}} =  \sum_\ell \frac{|A_{\ell}|}{|A_{0}|},
\end{equation}
where $|C|$ denotes the number of nonzeros in some sparse matrix $C$.
Note that this is equivalent to the total cost to perform one SpMV on each level of the hierarchy.
Using this, the CC is often considered to scale with OC\@.  For example, in the case of
a V{(2,2)} cycle, $\chi_{\textnormal{CC}} \approx 4\chi_{\textnormal{OC}}$. However, a more detailed model for CC
includes the residual computation and coarse-grid correction
steps.  While it is not typical to account for these parts of the solve phase,
they often contribute significantly to the CC, especially for the richer interpolation sparsity patterns examined later.
To this end, the CC for a V{($\nu_{\textnormal{pre}}$, $\nu_{\textnormal{post}}$)}-cycle is defined here as
\begin{equation}\label{eq:cc}
  \chi_{\textnormal{CC}} =  \sum_\ell \frac{    (\nu_{\textnormal{pre}} +\nu_{\textnormal{post}} + 1)|A_{\ell}| +
    |P_{\ell}| + |R_{\ell}|}{|A_{0}|},
\end{equation}
which reflects pre- and post-relaxation, a residual calculation, and one
interpolation and restriction per level (see solve phase discussion in
Section~\ref{sec:background}).

Detailed estimates of the complexity measures are often neglected in numerical results.
One contribution of this work is that precise estimates of the
SC, OC, and CC are provided for the numerical results presented in
Section~\ref{sec:numerical}. Coupled with the convergence factor,
this information is used to assess the effectiveness of the solver.
The SC estimates have been used to expose the expensive parts of the algorithm and
motivated the complexity reduction
techniques introduced in Section~\ref{sec:numerical:3DAni}.

\section{Root-node method}\label{sec:root-node}

The general algorithm for constructing a RN AMG\ hierarchy with $L+1$ levels,
using energy-minimizing interpolation smoothing, is given in
Algorithm~\ref{alg:rn_setup}.  The following subsections detail each
algorithmic step, comparing and contrasting with \saamg\ and \cfamg.
\begin{algorithm2e}[!ht]
   \DontPrintSemicolon   \KwIn{      \begin{tabular}[t]{l l}
      $A_0$:                 & fine-grid operator\\
      $B_0$:                 & fine-grid candidate vectors for $A$\\
      $\hat{B}_0$:                 & fine-grid candidate vectors for $A^T$ (if $A\neq A^T$)\\
      $d$:                      & interpolation sparsity pattern width\\
      \code{vector}:       & flag indicating vector-based problem\\
      \code{max\_size}: & threshold for max size of coarsest problem \\
      \code{prefilter}:     & Pre-filtering of interpolation sparsity pattern\\
      \code{postfilter}:    & Post-filtering of interpolation sparsity pattern
      \end{tabular}
   }
      \KwOut{      \begin{tabular}[t]{l}
      $A_{1}, \dots, A_{L}$,\\
      $P_{0}, \dots, P_{L-1}$,\\
      $R_{0}, \dots, R_{L-1}$
      \end{tabular}
   }
      $\ell = 0$\;
   \While{$\code{size}(A_\ell) > \code{max\_size}$}{            $S_{\ell}$        = \code{strength}{($A_{\ell}$)}\label{line:strength}          \tcc*[r]{Strength-of-connection of matrix}
      \If{\code{vector}}{        $S_\ell$ = \code{amalgamate}{($S_\ell$)}\label{line:amalgamate}               \tcc*[r]{Amalgate from degree-of-freedom to nodal}
      }
      $C_{\ell}$, roots = \code{aggregate}{($S_{\ell}$)}\label{line:aggregate}        \tcc*[r]{Construct aggregtes and root-nodes}
      \nonl\;
      $N_\ell \leftarrow S_{\ell}^d C_{\ell}$\label{line:sparsity}                    \tcc*[r]{Form interpolation sparsity pattern}
      \If{\code{prefilter}}{        $N_\ell$ = \code{filter}{($N_\ell$)}\label{line:prefilter}                    \tcc*{Eliminate small entries}
      }
      \If{\code{vector}}{        $N_\ell$, roots = \code{unamalgamate}{($N_\ell$, roots)}\label{line:unamal}   \tcc*[r]{From nodal to degree-of-freedom}
      }
      $N_\ell =$ \code{root\_node\_pattern}{($N_{\ell}$)}\label{line:pattern}         \tcc*[r]{Convert to root-node pattern}
      \nonl\;
            $B_{\ell}$           = \code{smooth}{($A_{\ell}$, $B_{\ell}$)}\label{Bimproveline}                         \tcc*[r]{Improve candidates with relaxation}
      $T_\ell, B_{\ell+1}$ = \code{inject}{($C_\ell$, $N_\ell$, $B_{\ell}$, roots)}\label{line:tent}              \tcc*[r]{Form tentative interpolation and $B_{\ell+1}$}
      $P_{\ell}$           = \code{improve}{($A_{\ell}$, $T_\ell$, $B_{\ell}$, $B_{\ell+1}$)}\label{line:improve} \tcc*[r]{Create smooth $P_\ell$ with $P_{\ell} B_{\ell+1} = B_{\ell}$}
      \If{\code{postfilter}}{         $P_\ell$    = \code{filter}{($P_\ell$)}\label{line:postfilter}                                           \tcc*{Eliminate small entries}
         $P_\ell$    = \code{enforce}{($P_\ell$, $B_{\ell}$, $B_{\ell+1}$ )}\label{line:enforce}                                           \tcc*{Enforce mode constraint with eqn. (\ref{eqn:Tproj})}
         $P_{\ell}$  = \code{improve}{($A_{\ell}$, $P_\ell$, $B_{\ell}$, $B_{\ell+1}$)}                           \tcc*[r]{Re-smooth $P$ with one iteration}
      }
                  \If{\code{symmetric}$(A_\ell)$}{         $R_\ell = P_\ell^T$
      }
      \Else{         $\hat{B}_{\ell}$                 = \code{smooth}{($A^T_{\ell}$, $\hat{B}_{\ell}$)}                       \tcc*[r]{Improve candidates for the non-symmetric case}
         $\hat{T}_\ell, \hat{B}_{\ell+1}$ = \code{inject}{($C_\ell$, $N_\ell$, $\hat{B}_{\ell}$, roots)}          \tcc*[r]{Form tentative restriction and $\hat{B}_{\ell+1}$}
         $R_\ell^T$ = \code{improve}{($A^T_{\ell}$, $\hat{T}_\ell$, $\hat{B}_{\ell}$, $\hat{B}_{\ell+1}$)}        \tcc*[r]{Create smooth $R_\ell$ with $R_{\ell}^T \hat{B}_{\ell+1} = \hat{B}_{\ell}$}
      \If{\code{postfilter}}{         $R^T_\ell$    = \code{filter}{($R^T_\ell$)}\label{line:Rpostfilter}                                           \tcc*{Eliminate small entries}
         $R^T_\ell$    = \code{enforce}{($R^T_\ell$, $\hat{B}_{\ell}$, $\hat{B}_{\ell+1}$ )}\label{line:Renforce}                                           \tcc*{Enforce mode constraint with eqn. (\ref{eqn:Tproj})}
         $R^T_{\ell}$  = \code{improve}{($A^T_{\ell}$, $R^T_\ell$, $\hat{B}_{\ell}$, $\hat{B}_{\ell+1}$)}                           \tcc*[r]{Re-smooth $R$ with one iteration}
         }
      }
            $A_{\ell+1}$ = $R_{\ell} A_{\ell} P_{\ell}$\label{line:galerkin}                                            \tcc*[r]{Form coarse-grid}
      $\ell = \ell+1$
   }
       \caption{\code{root\_node\_setup}}\label{alg:rn_setup}
\end{algorithm2e}

\subsection{Candidate vectors $B$}

As with \saamg, \textit{a priori} knowledge of the algebraically smooth error
is assumed as input in the form of a set of candidate vectors $B$. These
vectors are critical for ensuring accurate interpolation of important
algebraically smooth modes. In the case of $A$ being symmetric, one set
of candidate vectors, $B$, is sufficient. In the case of non-symmetric problems, a restriction
operator $R$ is formed independently (whereas $R = P^T$ in the case of symmetry)
through the \textit{left} candidate vectors $\hat{B}$, which target
smooth error in $A^T$. Generally if candidate vectors are not known or
provided, a constant vector is used, as the constant is geometrically smooth,
and a good choice for many problems. Finally, to ensure smooth (including at the boundaries)
candidate vectors, a small number of relaxation sweeps are applied to $A_\ell B_\ell = \mathbf{0}$
(line~\ref{Bimproveline}). This improves the algebraic smoothness
of $B_\ell$, especially near boundaries of a domain.  Even for textbook
examples such as a Laplacian, the standard candidate $B_\ell = {\bf 1}$ can
yield a poor approximation to algebraically smooth error near Dirichlet
boundaries.

\subsection{Strength matrix $S$}

The first level-specific step is the construction of an $n_\ell \times n_\ell$
SOC matrix, $S_\ell$ (line~\ref{line:strength}), which indicates strong
connections between DOFs in the problem. This matrix is used for the
aggregation of DOFs, and for the construction of a sparsity pattern,
$\mathcal{N}$, for $P$.  This work considers the classical strength
measure~\cite{RuStu1987}, symmetric strength measure~\cite{VaMaBr1996}, and
so-called evolution strength measure~\cite{OlScTu2009}, although other measures
have also been proposed~\cite{Brandt:2011tu,BrBrMaMaMcRu2006}.
The classical and symmetric strength measures essentially look at the magnitude
of an off-diagonal entry when determining if it represents a strong connection.
In contrast, the evolution measure computes strength around a DOF $i$ by
locally evolving a unit vector centered at $i$
with a few sweeps of weighted-Jacobi.  This creates a
locally smooth vector which is then post-processed to determine which matrix
entries in row $i$ are strong.
For instance, for anisotropic diffusion, the directions in which the unit vector
diffuses most quickly, are selected as strong connections.

After finishing, each method produces a matrix $S_{\ell}$ where individual
entries represent the strength-of-connection in the graph of $A_{\ell}$.  One modification to $S_{\ell}$
that is used here is to normalize each row so that
\begin{equation}
      \max_{j\ne i}{(S_\ell)}_{ij} = {(S_\ell)}_{ii} = 1. \label{eq:soc}
\end{equation}
That is, all elements are non-negative, and for each row in $S_\ell$, the diagonal and the largest
off-diagonal, i.e., the strongest connection, both equal 1. This scaling and the non-binary nature
of $S_\ell$ are important when computing the sparsity pattern in line~\ref{line:sparsity}.

For vector-based problems, $A_\ell$ has a block structure of block-size
$m\times m$ and for common cases such as elasticity, each block corresponds to
the DOFs associated with different variables but defined at the same spatial
node.  It is typical to group each block into a single
so-called \textit{supernode}~\cite{VaMaBr1996}, followed by aggregation only at the
supernode level.  This requires condensing the SOC with an \code{amalgamate}
step in line~\ref{line:amalgamate}, which reduces $S_\ell$ to an $n_\ell/m
\times n_\ell/m$ matrix. The amalgamated entry is equal to the maximal entry of
its associated $m \times m$ block in $S_\ell$.
This allows for aggregation based on supernodes.

\subsection{Aggregation}

The next operation is forming an aggregation, $\mathcal{A}_\ell$, as in \saamg\
(see Algorithm~\ref{alg:SAAMG}), and an associated list of root-nodes
(line~\ref{line:aggregate}).  The classical greedy aggregation
algorithm~\cite{VaMaBr1996} is used, wherein an unaggregated vertex in $S_\ell$
is selected (as the root-node) and all neighboring vertices with strong edge
weights are collected to form an aggregate.  An aggregation pattern matrix
$C_\ell$ is then defined as a partition of unity, such that $C_\ell(i,j) = 1$
if node $i$ belongs to aggregate $j$, and zero otherwise.  This pattern will be
used in forming the sparsity pattern for $P_\ell$. Other aggregation
routines such as pairwise aggregation~\cite{DAmbra:2013iwa,Notay:2010um}
have been considered, but have not demonstrated improvements in \rnamg\ performance, and
the choice of root node is less clear than in a greedy aggregation routine.

Example aggregates for standard isotropic diffusion problems are given in
Figure~\ref{fig:sample_agg}.  For the systems case, $m$
DOFs at each supernode implies that each root-node also contains
$m$ DOFs.
\begin{figure}[!ht]
     \centering
       \includegraphics[width=0.85\textwidth]{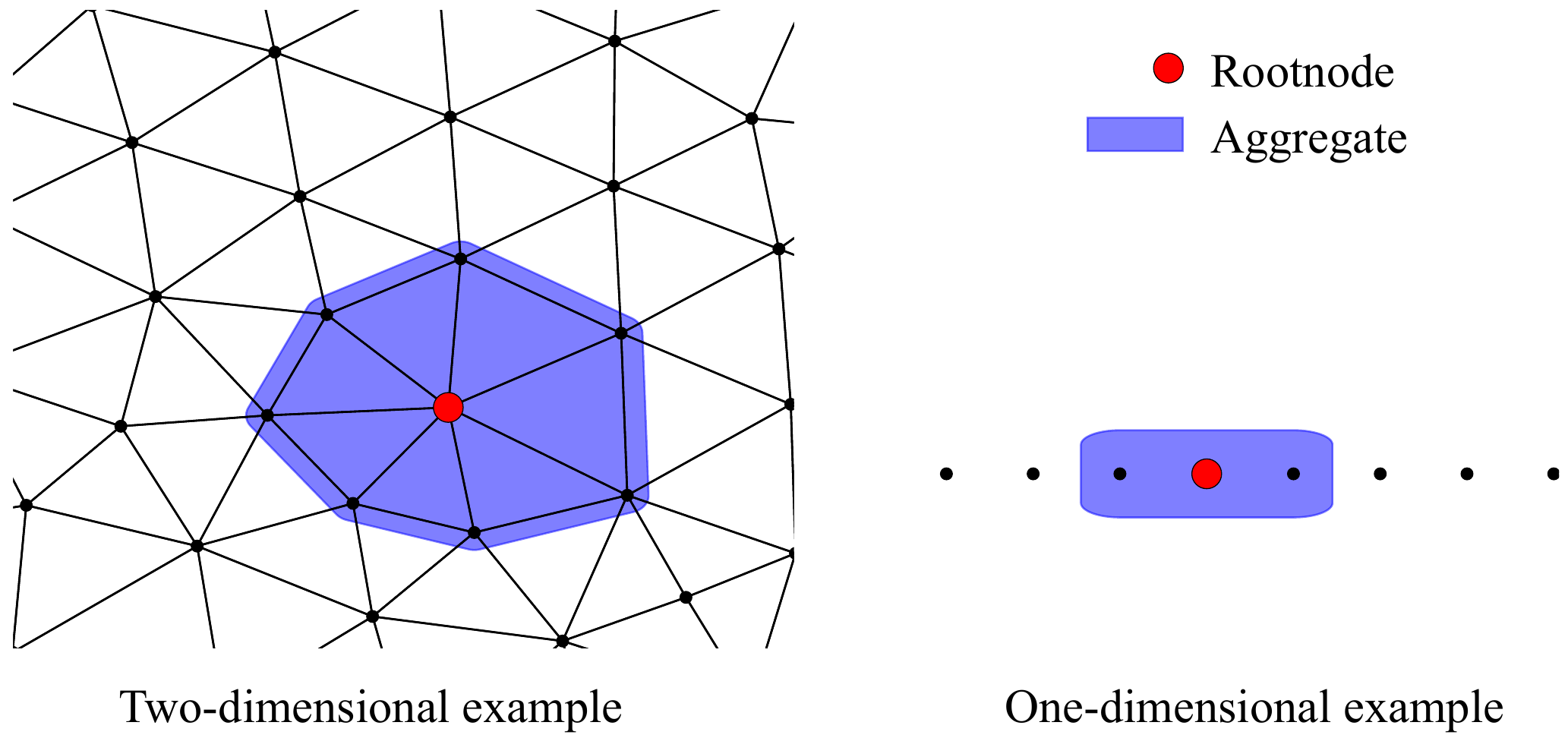}
         \caption{Sample aggregates on the finest level for standard isotropic diffusion in one and two dimensions.}\label{fig:sample_agg}
\end{figure}

\subsection{Arbitrary sparsity pattern $\mathcal{N}$}

Aggregation gives an interpolation structure, with
each root-node corresponding to one block column of $T$ and $P$ as
in~\eqref{eq:exampleinterp}, but not a sparsity pattern for $P$.
Thus, in Lines~\ref{line:sparsity}--\ref{line:pattern}, the sparsity pattern
$\mathcal{N}_\ell$ for interpolation is built, where
$(i,j) \not\in \mathcal{N}_\ell \Leftrightarrow N_\ell(i,j) = 0$.
Nonzero elements are based on growing the aggregation pattern
matrix, $C_\ell$, based on the strength matrix, $S_\ell$, through
multiplication $N_\ell = S_\ell^d C_\ell$, for some \textit{degree} or \textit{distance} $d$.
Using $d$ applications of $S_\ell$ extends the
interpolation stencil for a given F-point
to a distance of $d$ in the strength matrix.
Large values of $d$ allow
the sparsity pattern to grow in the direction of strong connections, allowing
for long-distance interpolation.  This approach differs from~\cite{OlScTu2011,
Sc2012} in that a normalized SOC matrix~\eqref{eq:soc} is used with
$C_\ell$, and as a result,
$N_\ell$ can be \emph{filtered} (line~\ref{line:prefilter}) by examining the magnitude of the
entries: larger entries indicate a stronger path from the root-node~$i$.
As a result, the normalization of $S_{\ell}$ imbues a relative size across columns so that the product with
the binary aggregation matrix $C_{\ell}$ yields individual entries related to strength-of-connection.

In the case of vector problems, the \texttt{unamalgamate} step in line~\ref{line:unamal} reverses amalgamation
and converts the list of root-nodes into a list of DOFs.  If the
amalgamated list of roots is of length $n_{\ell+1} / m$, then the
unamalgamated list is of length $n_{\ell}$.  Similarly, the unamalgamated entry
$N_{\ell}(i,j)$ is equal to the amalgamated entry
$N_{\ell}(\lfloor i/m \rfloor, \lfloor j/m \rfloor)$.

The last step in computing the sparsity is \texttt{root\_node\_pattern} in
line~\ref{line:pattern}.  This function enforces the root-node pattern of
interpolation in equation (\ref{eq:cfinterp}) by traversing $N_{\ell}$
to change each root-node row to be the corresponding row of the identity~---~i.e.,
the restriction to each root-node for the coarse-grid does not involve other root-nodes,
and each root-node is interpolated by value back to the fine grid.
Other than this root-node requirement,
\rnamg\ allows for arbitrary sparsity patterns and enables
selective control of the number of interpolation points.  This flexibility is used
in the next section, where the sparsity pattern is filtered (dropping entries) in a
way that targets only \emph{strong} long-distance connections in $P_{\ell}$.

It should be noted that the sparsity pattern for $R^T$ is built in the same manner as
$P$: by expanding a tentative operator with a SOC matrix based on $A$.
Another option is to construct the sparsity pattern for $R$ based on
a SOC from $A^T$; however, numerical results consistently indicate a degradation
in convergence factors when doing so.
Choosing the \textit{optimal} sparsity pattern for both $P$ and for $R$ remains an important and open research question.

\subsection{Filtering sparsity pattern}
\label{sec:filtering}

A filtering step, which removes nonzeros in $N_{\ell}$, is used after the
construction of the sparsity pattern for $P$.  A large degree $d$~(see
line~\ref{line:sparsity}) is often needed~\cite{Sc2012, KeMaSc2015} to
construct effective interpolation operators; filtering can limit the
additional cost due to the growth in the sparsity pattern. In particular,
filtering allows for long-distance interpolation in the direction of strong connections,
while limiting complexity.

Pre-filtering is used to filter the sparsity
pattern matrix $N_\ell$ before energy-minimization.
Here, entries are eliminated in $N_\ell$ prior to constructing $P_\ell$ based on the
size of entries in $N_\ell = S_\ell^d C_\ell$, which indicate the
strength-of-connection between two DOFs. Because pre-filtering is only based
on SOC and not the fully formed interpolation operator, it is possible that influential
entries are inadvertently removed, thus degrading convergence. However,
in practice trimming the sparsity pattern of $P_\ell$ before initiating the construction
significantly lowers the SC in many cases, with minimal impact on AMG convergence.

Given an initial sparsity pattern, entries are filtered as in~\cite{DeFaNoYa_2008}, by either
retaining the $k$ largest values in a row or by applying a drop
tolerance $\theta$.
Algorithm~\ref{alg:filter}
describes this process in detail, where
$\max(G,i,k)$ is the $k$th largest off-diagonal entry in row $i$.
The idea of pre-filtering has shown to be effective for model problems
using a polynomial approximation to $A_{ff}^{-1}$ in~\cite{Brannick:2007fb}. In contrast, the pre-filtering
used here is less expensive and relies on values already computed by the root-node algorithm.
\begin{algorithm2e}[!ht]
  \caption{\texttt{filter{$(G)$}}}\label{alg:filter}
   \DontPrintSemicolon   \KwIn{      \begin{tabular}[t]{l l}
      $G$:                 & matrix to be filtered \\
      $\theta$:            & filtering drop-tolerance \\
      $k$:                 & filtering threshold \\
      \end{tabular}
   }
   \KwOut{      \begin{tabular}[t]{l}
      $G$
      \end{tabular}
   }
   \If{$k$}{     \For{$|G(i,j)| < \max(G,i,k)$}{       $G(i,j) \leftarrow 0$\label{eq:filter_k}
     }
   }
   \If{$\theta$}{     \For{$|G(i,j)| < \theta \max(G,i,1)$}{     $G(i,j) \leftarrow 0$\label{eq:filter_theta}
   }
   }
\end{algorithm2e}

\subsection{Interpolation construction}
\label{sec:interp}

A \textit{tentative} interpolation operator, $T$, is constructed from the sparsity pattern.
The full interpolation is then formed based on a constrained energy-minimization
with the following principles:
\begin{enumerate}[label={\bf\roman*.}]
  \item $T$ and $P$ satisfy interpolation constraints of provided algebraically smooth
  	candidates. That is,
    \begin{equation}\label{eqn:mode_constraint}
        T B_c  = B \quad \textnormal{and} \quad P B_c  = B,
    \end{equation}
    for candidates $B$, and coarse level candidates $B_c$.
  \item The \texttt{improve} procedure reduces the \textit{energy} of each $j$th column of $P$:
    \begin{equation}
      \|P_{(j)}\| \leq \|T_{(j)}\|,
    \end{equation}
    for some $A$-induced norm $\|\cdot\|$.  For instance, interpolation smoothing in \saamg\ is one example.
  \item Given sparsity pattern $\mathcal{N}$,
    \begin{equation}
      T_{ij} = 0\quad\textnormal{if $(i,j) \notin\mathcal{N}$} \quad \textnormal{and} \quad       P_{ij} = 0\quad\textnormal{if $(i,j) \notin\mathcal{N}$}.     \end{equation}
\end{enumerate}
The energy-minimization approach~\cite{OlScTu2011} is used here, which
satisfies these principles.
Energy-minimization is an iterative smoothing process that improves $P$,
through several passes.  As a result, growth in the sparsity
pattern of $P$ necessitates a constraint on the sparsity pattern
constraint, $\mathcal{N}$.  However, in enforcing $\mathcal{N}$ by dropping
entries in $P$, the constraints are no longer satisfied.  In response, the constraints are
enforced as an additional step.

Root-node AMG proceeds by taking the aggregation and list of root-nodes to construct the
coarse-grid candidates by injection, $B_{\ell+1}(i,j) = B_{\ell}(k,j)$, where $k$
is the $i$th root-node. If $m$ is the block size of the original matrix ($m=1$ for
a scalar problem), then \textit{only} the first $m$
candidates are injected over each aggregate to form an initial
$T_\ell$ (line~\ref{line:tent}). As a result, each root-node represents $m$ DOFs on the coarse grid.
An additional step is performed on $T_\ell$, normalizing each column so that the coarse-grid
variables inject to the fine-grid root-nodes. This process yields the following form
\begin{equation}\label{eq:root-nodeform}
  T_\ell =
  \begin{array}{c@{}c}
    \left[
      \begin{array}{c}
        W_\ell \\
        I\\
      \end{array}
    \right]
  &
  \begin{array}{l}
    \} \,\,\textnormal{Non Root-nodes}\\
    \} \,\,\textnormal{Root-nodes}\\
  \end{array}
  \end{array}.
\end{equation}
For $m=1$, $T_\ell$ has non-overlapping columns; for $m>1$, $W_\ell$ is block diagonal,
as $T$ in \eqref{eq:exampleinterp}. With the identity over C-points in $T_{\ell}$, \rnamg\ 
resembles that of \cfamg\ (cf.~\eqref{eq:cfinterp}). 

If there are more than $m$ candidates, the remaining candidates are projected into
$\sspan(T_\ell)$ in the Euclidean inner-product. In \rnamg, it is assumed that the
sparsity pattern has sufficient DOFs that this is an underdetermined problem.
This results in the minimal norm update to each row of $T_\ell$ such that
$T_\ell B_{\ell+1} = B_\ell$ and $T_\ell$ obeys the sparsity pattern
$\mathcal{N}_\ell$.  More specifically, an update $u$ to only the
allowed nonzero portion of row $i$ of $T_\ell$ (called $t$) is computed
by solving
\begin{equation}\label{eqn:Tproj}
   (t + u)B_{\ell+1} = B_\ell \quad \Leftrightarrow \quad B_{\ell+1}^T u^T = B_\ell^T - B_{\ell+1}^T t^T,
\end{equation}
for $u$ using least-squares. Note that only $W_\ell$ is modified in \eqref{eqn:Tproj}, that is,
injection over C-points is maintained and $W_\ell$ is expanded to interpolate candidate vectors.

Next, interpolation $P_\ell$ is formed using \code{improve} in
line~\ref{line:improve}.
Energy-minimization forms  $P_\ell$ as a succession of energy-minimization
updates to $T_\ell$.  Each update $U$ is computed to reduce the energy of each
column via a Krylov process and to also satisfy $U B_{\ell+1} = 0$.  As a result, (i) $(T_\ell + U)
B_{\ell+1} = T_\ell B_{\ell+1}$ (constraints are satisfied exactly),
and (ii) $\|T_\ell + U\| \leq \|T-\ell\|$ (energy is reduced).
The projection operator enforcing the constraints is analogous to~\eqref{eqn:Tproj}~(see~\cite{OlScTu2011}).

\subsection{Filtering interpolation}
\label{sec:post-filtering}

After $P_\ell$ is formed, a post-filtering process (similar to pre-filtering in Section~\ref{sec:filtering})
is applied in line~\ref{line:postfilter} to reduce complexity. Post-filtering removes elements directly
from $P_\ell$ after smoothing, but this leads to a $P_{\ell}$ that violates the mode interpolation 
constraints \eqref{eqn:mode_constraint}.  Thus, the function \code{enforce} is used to re-apply 
these constraints via \eqref{eqn:Tproj}.  
Finally, an additional iteration of \code{improve} is used to account for large increases in energy caused by
removing entries. One advantage of post-filtering is that element removal is based on the
actual smoothed interpolation stencil entries, and is thus less likely to inadvertently
degrade convergence compared with pre-filtering~---~i.e.\ by removing important entries
from $P$. Post-filtering generally results in a lower complexity in the Galerkin coarse-grid operator
(line~\ref{line:galerkin}) and all subsequent coarser grid operations.
A similar filtering approach is also effective for classical AMG methods~\cite{DeFaNoYa_2008}.
However, post-filtering does not reduce the OC and SC
as effectively as pre-filtering, in particular because energy-minimization (one of the dominant
costs in SC) is applied to a larger sparsity pattern, which is then trimmed afterwards.

\subsection{Coarse-grid construction}

The final step in Algorithm~\ref{alg:rn_setup} is the construction of the coarse-grid
operator through a (Petrov) Galerkin triple-matrix product $A_{\ell+1}=R_\ell A_\ell P_\ell$.  For
symmetric $A_\ell$, restriction is the transpose of interpolation.
For non-symmetric $A_\ell$, the interpolation construction process is
duplicated for $R_{\ell}^T$ using $A_\ell^T$ (the $\hat{\cdot}$ notation is
used to denote the quantities used to compute $R_{\ell}$)~\cite{BrMaMcRuSa2010}.

\subsection{Discussion: Compare and contrast \rnamg\ with \saamg}

Similar to \saamg, \rnamg\ facilitates the use of
multiple, arbitrary candidate vectors, but handles the complexity challenges by
not adding columns to $P$ for each additional candidate. The dimensions of $T$
and $P$ are fixed, with the number of columns in $T$ and $P$ being equal to
the number of aggregates.  The candidates $B$ are projected exactly into $\sspan(P)$ and
$\sspan(T)$, assuming a sparsity pattern with enough entries.  With
\saamg, for each candidate vector added to $B$, a new column is added to $T$
and $P$ for every aggregate (see~\eqref{eq:exampleinterp}). This difference allows \rnamg\ to
exhibit significantly lower complexity than \saamg\ in some instances.

Root-node AMG also uses interpolation smoothing, like \saamg, so that $P$ is
iteratively improved. However, standard interpolation smoothing does
\textit{not} satisfy the exact candidate interpolation constraints in general,
although it does attempt to target the same three principles outlined in
Section~\ref{sec:interp} (including accurate but not exact candidate vector interpolation).
A consequence of satisfying the constraints is that the candidate vectors are
exactly represented in the range of interpolation.  To measure this,
consider the error in interpolating $B$ found through an orthogonal projection
of $B$ into the range of interpolation:
\begin{equation} e_{B} = (I - P P^{\dagger})B, \end{equation}
where $P^{\dagger}$ is the pseudo-inverse of $P$.  The point-wise values of the
error $e_B$ are shown in Figure~\ref{fig:agg1d-1-CSR-Berr}, where \rnamg\
achieves much lower error, directly satisfying the constraints (globally).
\begin{figure}[!ht]
   \centering
   \includegraphics[width=0.48\textwidth]{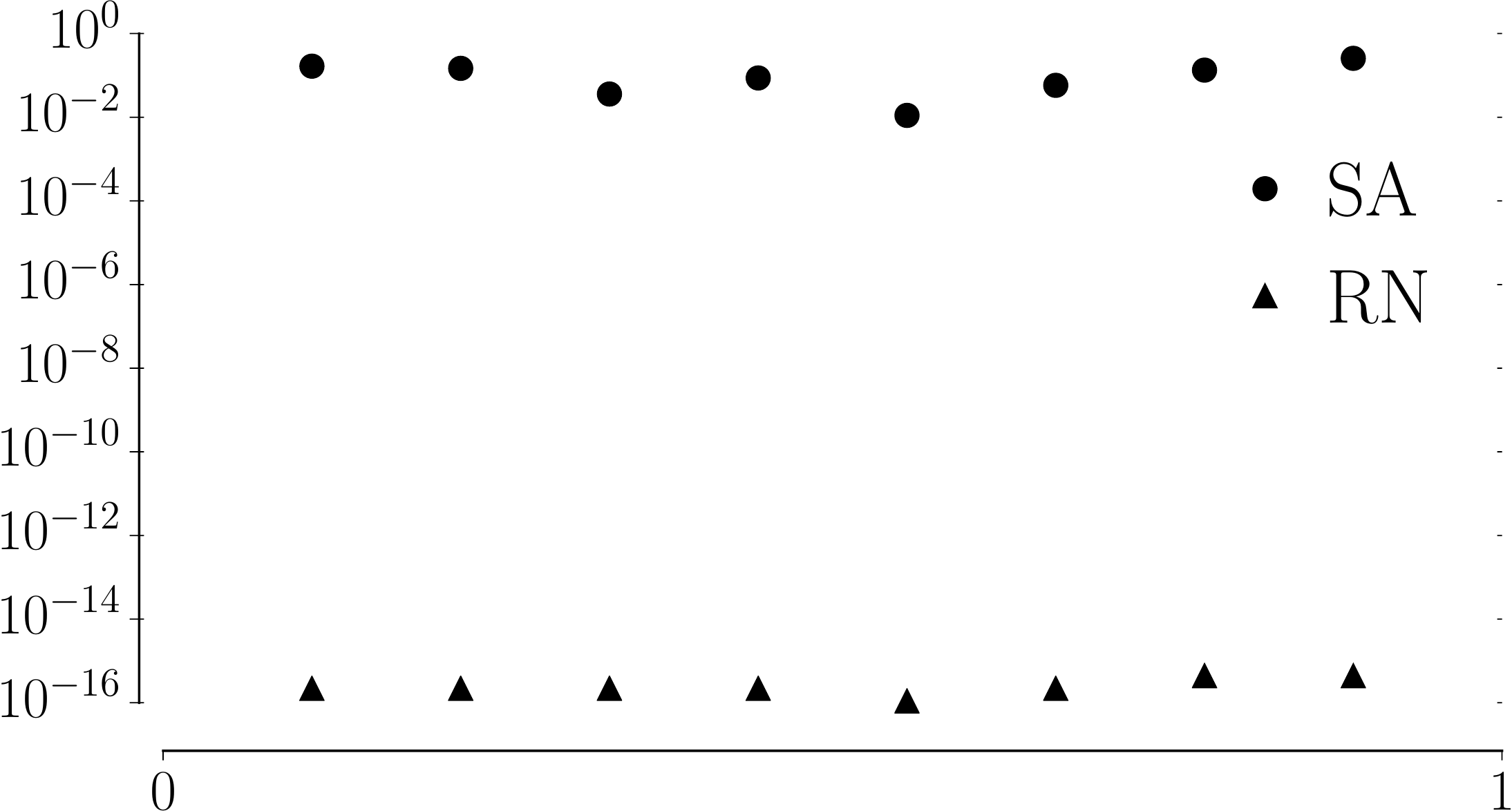}
   \caption{Error $e_B$ in satisfying the constraints $B \equiv {\bf 1} \in
      \sspan(P)$ for a 1D Poisson problem with eight DOFs using central finite
    differences.}\label{fig:agg1d-1-CSR-Berr}
                \end{figure}

Additionally, standard interpolation smoothing does not allow for multiple
smoothing passes (and hence longer-distance interpolation) without suffering
from complexity issues and fill-in in $P$.  The arbitrary sparsity
pattern allows \rnamg\ to have longer-distance interpolation and multiple smoothing iterations,
while also effectively managing complexity.

An \saamg-type method supporting vector problems can be derived from
Algorithm~\ref{alg:rn_setup} by removing the root-node specific lines and by
using the \texttt{inject} and \texttt{smooth} functions from
Algorithm~\ref{alg:SAAMG} to form $T_\ell$ and $P_\ell$.  This form of \saamg\
also supports non-symmetric problems~\cite{Sala:2008cv}.

\subsection{Discussion: Compare and contrast \rnamg\ with \cfamg}

Compared with \cfamg, \rnamg\ adopts the structure of $P$ from
(\ref{eq:cfinterp}), which aids the use of pre- and post-filtering since the columns of interpolation
are normalized (with the identity).
(In Section~\ref{sec:theory}, additional theoretical motivation for
using this form is investigated.)  To understand this, consider how the CF AMG
structure~\eqref{eq:root-nodeform} scales interpolation around each root-node
for a simple 1D Laplace example on eight nodes with finite differencing
and $B={\bf 1}$.  In this case, three aggregates are formed  as in
Figure~\ref{fig:agg1d-setup}.  Each column of $T$ is shown in
Figure~\ref{fig:agg1d-1-CSR-RN-T}, where a relative balance in weighting across
aggregates is observed, in contrast to \saamg, which is shown in
Figure~\ref{fig:agg1d-1-CSR-SA}.  This balance in weighting is evident again in
the final $P$ in Figure~\ref{fig:agg1d-1-CSR-RN-P}, again because the identity
form is preserved according to equation~\eqref{eq:root-nodeform}.

This scaling aids the filtering strategy, which eliminates relatively small interpolation
weights in the rows of $P$, because it provides the rationale for
comparing the magnitude of entries across columns.  In other words, \rnamg\ stipulates that a
large interpolation weight is one, representing injection from the coarse to
fine grid, visible as the peaks in Figure~\ref{fig:agg1d-1-CSR-RN-P}.
\begin{figure}[!ht]
  \centering
  \begin{subfigure}[b]{0.48\textwidth}
    \includegraphics[width=\textwidth]{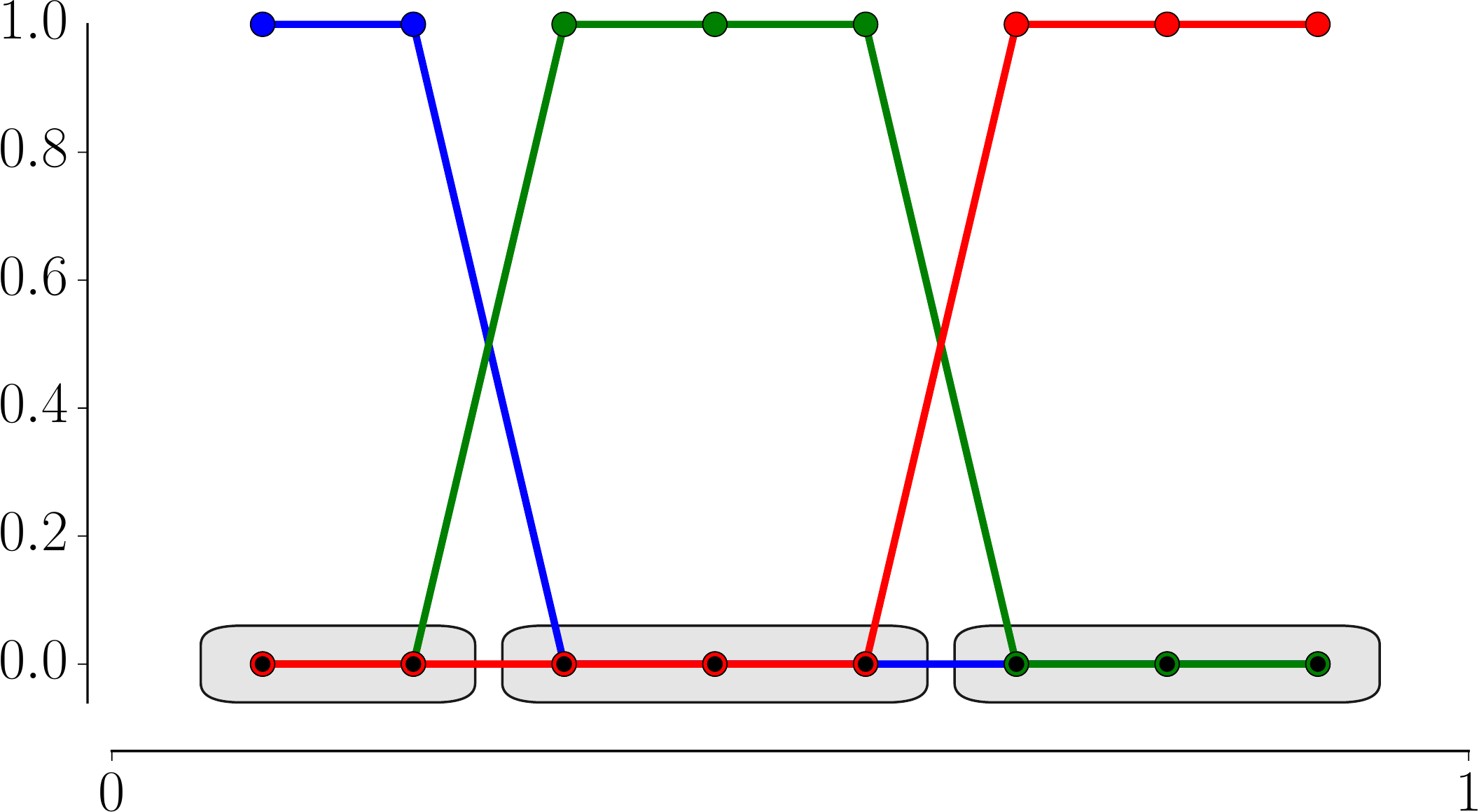}
    \caption{Columns of tentative interpolation $T$.}\label{fig:agg1d-1-CSR-RN-T}
  \end{subfigure}
  \hfill
  \begin{subfigure}[b]{0.48\textwidth}
    \includegraphics[width=\textwidth]{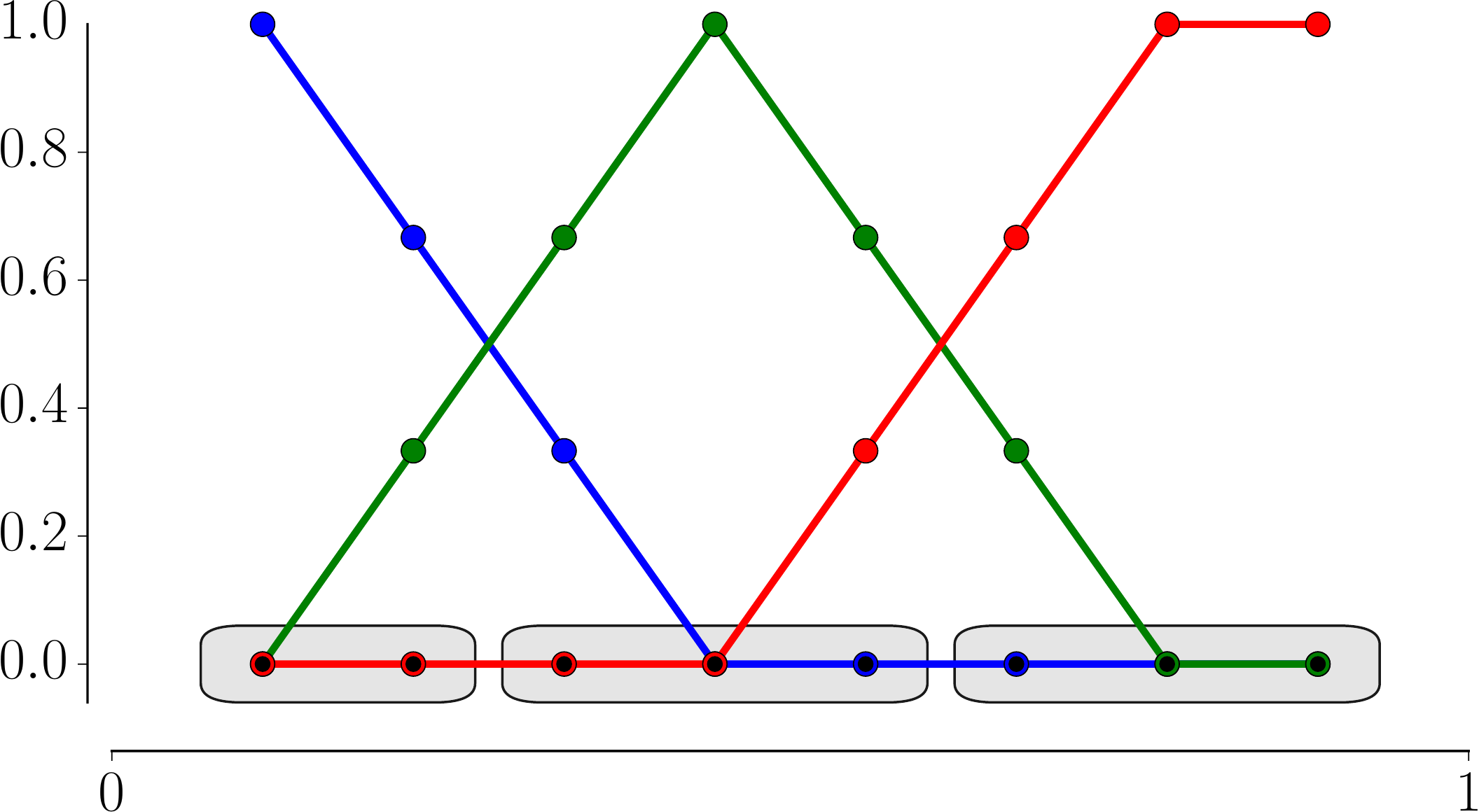}
    \caption{Columns of interpolation $P$.}\label{fig:agg1d-1-CSR-RN-P}
  \end{subfigure}
  \caption{Interpolation using a single candidate $B={\bf 1}$ in \rnamg.}\label{fig:agg1d-1-CSR}
\end{figure}

In contrast to \cfamg, there are the many similarities that \rnamg\ shares with
\saamg, with the ability to iteratively improve $P$ and guarantee accurate
interpolation of the user-defined candidate vectors $B$ being key differences.
Regarding similarities to more recent adaptive CF-style AMG methods,
bootstrap AMG (BAMG)~\cite{BrBrKaLi2011} also fits multiple candidate
vectors into the range of interpolation, but does so by generating a set of many candidates
and over-determining each row of $P$.  In this respect, \rnamg\
is closer to \saamg, as each row of $P$ is under-determined and only
fits a small number of candidate vectors into the range of $P$, using interpolation
smoothing to fully determine each interpolation entry.

In contrast to both \cfamg\ and \saamg, \rnamg\ allows for automatically
expanded sparsity patterns $\mathcal{N}_\ell$, which can be filtered, to
facilitate long-distance interpolation. This is critical for robustness and
performance, as shown in the examples in Section~\ref{sec:numerical}, and is a
unique feature of \rnamg.

\section{Theoretical motivation for root-node}\label{sec:theory}

In this section, theoretical motivation for RN AMG is introduced.
Initially, it is assumed that $A$ is SPD in order
to connect with classical AMG theory. Based on relations established in the
symmetric setting, some results are extended to non-symmetric systems in Section~\ref{sec:theory:nonsymm}.

\subsection{The symmetric case}

Let $A\in\mathbb{R}^{n\times n}$ be SPD, $\|\cdot\|$ and $\|\cdot\|_A$ represent the
$l^2$- and $A$-norms, respectively, and $P\,:\, \mathbb{R}^{n_c}\rightarrow\mathbb{R}^n$
be an interpolation operator defining a coarse space of size $n_c$.  The error propagation
operator for a two-grid method is given by
\begin{equation}
  E_{\textnormal{TG}} = I - P {(P^T A P)}^{-1} P^T A,
\end{equation}
and a multilevel version, $E_{\textnormal{MG}}$ is similarly defined~(see~\cite{2014MaOlamgtheory}).
Here, bounds on the $A$-norm of $E_{\textnormal{TG}}$ and
$E_{\textnormal{MG}}$ are constructed.
The \textit{weak approximation property} (WAP) gives necessary and sufficient
conditions for two-grid convergence as follows: there exists $K\in\mathbb{R}$ such that for any vector $\mathbf{u} \in \mathbb{R}^n$,
\begin{equation}\label{eq:wap}
\min_{\mathbf{w}_c\in\mathbb{R}^{n_c}} \|\mathbf{u} - P \mathbf{w}_c\|^2 \leq \frac{K}{\|A\|} \|\mathbf{u}\|_A^2,
\end{equation}
wherein $\|E_{TG}\|_A = 1 - \frac{1}{K_{TG}}$ and $1\leq K_{TG} \leq K$~\cite{Va2008}. For simplicity,~\eqref{eq:wap} is based on Richardson relaxation. A WAP with respect
to a general relaxation scheme along with a tight bound on $K_{TG}$ can be found in~\cite{FaVa2004,Va2008}. The \textit{strong approximation property} (SAP)
establishes multilevel convergence with a stronger
condition: there exists $K\in\mathbb{R}$ such that for any vector $\mathbf{u} \in \mathbb{R}^{n}$,
\begin{equation}\label{eq:sap}
\min_{\mathbf{w}_c\in\mathbb{R}^{n_c}} \|\mathbf{u} - P \mathbf{w}_c\|_A^2 \leq \frac{K}{\|A\|} \|A \mathbf{u}\|^2.
\end{equation}
If~\eqref{eq:sap} holds on each level of the hierarchy, then $\|E_{MG}\|_A = 1 - \frac{1}{K_{MG}}$
and $1\leq K_{MG} \leq 1 + K\frac{\|M\|}{\|A\|}$,
where $M$ is the chosen relaxation scheme of the form $\mathbf{x}_{k+1} = \mathbf{x}_k +
M^{-1}\mathbf{r}_k$, for residual $\mathbf{r}$~\cite{Va2008}.

Since $A$ is assumed to be SPD, its eigenvectors form an $l^2$- and $A$-orthonormal
basis for the space $\mathbf{R}^n$. Thus, if the WAP and SAP hold for all eigenvectors,
they hold for all vectors, and it follows that
the WAP requires eigenvectors be interpolated with accuracy on the order
of the corresponding eigenvalue, and the SAP requires interpolation accuracy on the order
of the eigenvalue squared. This leads to an equivalence of satisfying the WAP
based on $A^2$ and the SAP for $A$ as follows:
\begin{lemma}[Lemma 5.20~\cite{Va2008}]\label{th:wap_sap}
Let $A \in\mathbb{R}^{n \times n}$ be SPD and $P \in \mathbb{R}^{n \times n_c}$. Then
\begin{align}
\min_{\mathbf{w}_c\in\mathbb{R}^{n_c}} \|\mathbf{u} - P \mathbf{w}_c\|^2 &\leq \frac{K^2}{\|A^2\|} \|\mathbf{u}\|_{A^2}^2 \hspace{3ex}\text{for all }\mathbf{u},\\
\intertext{if and only if}
\min_{w_c} \|\mathbf{u} - P \mathbf{w}_c\|_A^2 &\leq \frac{K}{\|A\|} \|A \mathbf{u}\|^2 \hspace{3ex}\text{for all }\mathbf{u}.
\end{align}
\end{lemma}

The accuracy demands of the WAP and SAP with respect to eigenvalues
indicates that the range of $P$ should contain
eigenvectors of $A$ associated with small eigenvalues (or so-called \textit{algebraically smooth} modes).
In building AMG hierarchies, this generally takes one of two forms, (i) ensuring that known
low-energy modes are exactly represented in the range of $P$, and (ii) minimizing columns of $P$ in the
$A$-norm so that the range of $P$ corresponds to
algebraically smooth vectors. Recall that RN AMG combines each of these approaches in
a constrained energy-minimization~\cite{Brannick:2007fb,OlScTu2011,Sc2012,2011_OlSc_scidac11}.

Given a CF-splitting of the current grid, consider a matrix ordering of the form $A =
\begin{bmatrix} A_{ff} & A_{fc}\\ A_{cf} & A_{cc}\end{bmatrix}$, where $A_{ff}$ corresponds
to F-point-to-F-point connections, $A_{fc}$ to F-point-to-C-point connections, and so on.
In a CF AMG context, interpolation is then assumed to have the form $P =
\begin{pmatrix} W\\I\end{pmatrix}$, where the lower identity block interpolates and restricts
C-points by injection and $W$ interpolates and restricts $F$-points based on
linear combinations. The minimizing coarse-grid vector, $\mathbf{w}_c$, in
the WAP and the SAP is given by $l^2$-orthogonal and $A$-orthogonal projections of the
vector $\mathbf{u}$ onto the range of $P$, respectively. In practice, such projections
are generally too expensive to form explicitly; thus, computable measures are also of
interest. One option consistent with \cfamg\ is to let $\mathbf{w}_c = \mathbf{u}_c$, that
is define $\mathbf{w}_c$ as the restriction of $\mathbf{u}$ to C-points. This provides a
bound on the WAP, as $\min_{\mathbf{w}_c\in\mathbb{R}^{n_c}} \|\mathbf{u} - P \mathbf{w}_c\|^2
\leq  \|\mathbf{u} - P \mathbf{u}_c\|^2$ for all $\mathbf{u}$, and thus
\begin{align}\label{eq:mu}
\mu(P) := \max_{\mathbf{u}\neq 0} \frac{\|\mathbf{u} - P \mathbf{u}_c\|^2}{\|\mathbf{u}\|_A^2} \geq K_{TG}.
\end{align}
Assuming $P = \begin{pmatrix} W\\I\end{pmatrix}$, the optimal interpolation operator
under $\mu(P)$ is given by
\begin{align}\label{eq:ideal}
P_{\textnormal{ideal}} & = \argmin_{P} \max_{\mathbf{u}} \frac{\|\mathbf{u} - P \mathbf{u}_c\|^2}{\|\mathbf{u}\|_A^2} \nonumber\\
& = \begin{bmatrix} -A_{ff}^{-1}A_{fc} \\ I\end{bmatrix},
\end{align}
where $P_{\textnormal{ideal}}$ is referred to as ``ideal interpolation.''

In addition to $P_{\textnormal{ideal}}$ being optimal with respect to the measure $\mu(P)$ (and
thus satisfying the WAP for some $K \geq K_{TG}$), if $A_{ff}$ is well-conditioned, then
$\|A_{ff}^{-1}\|$ is bounded by a small constant, in which case
$P_{\textnormal{ideal}}$ also satisfies the SAP\@:
\begin{align*}
\min_{\mathbf{w}_c\in\mathcal{V}_c} \|\mathbf{u} - P_{\textnormal{ideal}}\mathbf{w}_c\|_A^2 & \leq
	\|\mathbf{u} - P_{\textnormal{ideal}}\mathbf{u}_c\|_A^2
\leq  \|A_{ff}^{-1}\| \|A\mathbf{u}\|^2
\leq \frac{K}{\|A\|}\|A\mathbf{u}\|^2,
\end{align*}
for some $K$. While $P_{\textnormal{ideal}}$ indicates an effective interpolation scheme,
$A_{ff}^{-1}A_{fc}$ is often a dense matrix and difficult to compute. However, if $A_{ff}$
is well-conditioned, its entries decay exponentially fast away from the diagonal~\cite{Brannick:2007fb}, suggesting that a sparse approximation can be formed.
In addition to directly satisfying constraint vectors, the energy-minimization
piece of root-node interpolation constructs a sparse approximation to $P_{\textnormal{ideal}}$,
with the goal of retaining the convergence properties of ideal interpolation, while limiting
coarse-grid complexity.
In general, aggregation-based AMG is motivated through energy-minimization principles
over the columns of the interpolation operator~\cite{Janka:1999tc,Van:2001bw,Napov:2011tj,
Brezina:2012kl,Chen:2015vx,hu:2016}. The root-node approach, however, is supported by a more
complete theoretical motivation, as well as practical benefits including the use of aggregation
or CF-splittings for coarsening, and better conditioning of the coarse-grid operator in the
non-symmetric setting (Section~\ref{sec:theory:nonsymm}).

Observe that $A P_{\textnormal{ideal}} = \begin{pmatrix} 0 \\ S \end{pmatrix}$,
where $S=A_{cc} - A_{cf}A_{ff}^{-1}A_{fc}$ is the Schur complement of $A$. Given that
$AP_{\textnormal{ideal}} = 0$ over F-points, this motivates minimizing columns
of $P = \begin{pmatrix}W\\I\end{pmatrix}$ in the $A$-norm to approximate the action
of $P_{\textnormal{ideal}}$. The identity block over C-points along with any
constraints enforced ensure that columns of $W$ are nonzero (the solution to
minimizing a general $P$ in the $A$-norm without constraints is $P = \mathbf{0}$).
Coupled with a predetermined sparsity pattern and constraints $PB_c = B$, where
$B$ is a set of column-wise constraint vectors to be in the range of $P$ and $B_c$
$B$ restricted to C-points, energy minimization in the $A$-norm is exactly the conjugate
gradient variant of energy minimization proposed in~\cite{OlScTu2011}.
Lemma~\ref{th:cg} shows the relationship between $P_{\textnormal{ideal}}$ and
energy minimization. That is, the CG variant of energy minimization used to form $P$
over a given sparsity pattern is equivalent to minimizing the difference between
columns of $P$ and $P_{\textnormal{ideal}}$ in the $A$-norm, over a given sparsity
pattern.
\begin{lemma}\label{th:cg}
Let $A \in\mathbb{R}^{n \times n}$ be SPD, $P_{\textnormal{ideal}} \in \mathbb{R}^{n \times n_c}$
be given by~\eqref{eq:ideal}, and $\mathbf{e}_\ell$ the $\ell$th canonical basis vector, where
$\mathbf{p}_\ell = P\mathbf{e}_\ell$ is the $\ell$th column of $P$. Denote by $\mathcal{N}^{F}$ a
sparsity pattern for any matrix $W\in\mathbb{R}^{n_f \times n_c}$ where $W_{ij}\ne 0$ if
$(i,j) \in \mathcal{N}^F$. Then define the set $\mathcal{P}_\ell$ as the $\ell$th column
of any matrix with the structure of $P_{\textnormal{ideal}}$, restricted to
sparsity pattern $\mathcal{N}^F$, that is
\begin{equation}
\mathcal{P}_\ell = \bigg\{ P\mathbf{e}_l \, : \, P = \begin{bmatrix} W\\ I\end{bmatrix},
\textnormal{where $W \in\mathbb{R}^{n_f \times n_c}$ and $W_{ij} = 0$ if ${(i,j)} \not\in \mathcal{N}^{F}_{ij}$}\bigg\}. \end{equation}
Then for $l=1,\dots,n_c$,
\begin{equation}\label{eq:mins}
  \argmin_{\mathbf{p}_\ell\in\mathcal{P}_\ell} \left\|\mathbf{p}_\ell\right\|_{A} =
  	\argmin_{\mathbf{p}_\ell \in\mathcal{P}_\ell} \left\| \mathbf{p}_\ell - P_{\textnormal{ideal}}\mathbf{e}_\ell\right\|_A.
\end{equation}
Equivalently, minimizing the columns of $P$ in the $A$-norm is equivalent to minimizing
the difference between columns of $P$ and $P_{\textnormal{ideal}}$ in the $A$-norm.
\end{lemma}
\begin{proof}
Equivalence is established by demonstrating identical weak forms for the two
minimization problems in~\eqref{eq:mins}. Consider $\mathbf{p}_\ell \in\mathcal{P}_\ell$ and
define $N_\ell$ to be the diagonal matrix that enforces sparsity pattern $\mathcal{N}^F$ on the
$\ell$th column of $W$, $\mathbf{w}_\ell$. That is, $\mathbf{w}_\ell = W \mathbf{e}_\ell= N_\ell W \mathbf{e}_\ell$, where the $k$th entry of $\mathbf{w}_\ell$ equals zero if $(k,\ell)\not\in\mathcal{N}^F$.
\begin{enumerate}
\item Consider minimizing the $l$th column of $P$, given by $\mathbf{p}_\ell =
\begin{bmatrix}\mathbf{w}_\ell\\ \mathbf{e}_\ell\end{bmatrix}$, in the $A$-norm. To this end, define the functional
$G(\mathbf{w}_\ell) = \Big\langle A \mathbf{p}_\ell, \mathbf{p}_\ell\Big\rangle
= \Big\langle A\begin{bmatrix}\mathbf{w}_\ell\\ \mathbf{e}_\ell\end{bmatrix},
               \begin{bmatrix}\mathbf{w}_\ell\\ \mathbf{e}_\ell\end{bmatrix}\Big\rangle$, with first variation
\begin{align*}
G'(\mathbf{w}_\ell ; \mathbf{v}) & = 2\Big\langle N_\ell A_{ff} N_\ell \mathbf{w}_\ell  - N_\ell A_{fc} \mathbf{e}_\ell, \mathbf{v}\Big\rangle
\end{align*}
for $\mathbf{v} = N_\ell \mathbf{v}$.
The weak form for of minimizing $G$ is then given by
\begin{align*}
N_\ell A_{ff} N_\ell \mathbf{w}_\ell = N_\ell A_{fc} \mathbf{e}_\ell.
\end{align*}

\item Consider minimizing the difference between the $l$th column of $P$ and $P_{\textnormal{ideal}}$ in the $A$-norm. That is, define the functional
\begin{align*}
H(\mathbf{w}_\ell) & = \Big\langle A(P - P_{\textnormal{ideal}})\mathbf{e}_\ell , (P - P_{\textnormal{ideal}})\mathbf{e}_\ell \Big\rangle \\
       & = \Big\langle A_{ff}(N_\ell \mathbf{w}_\ell - A_{ff}^{-1}A_{fc} \mathbf{e}_\ell), N_\ell \mathbf{w}_\ell - A_{ff}^{-1}A_{fc}\mathbf{e}_\ell\Big\rangle.
\end{align*}
Taking the first variation yields
\begin{align*}
H'(\mathbf{w}_\ell; \mathbf{v}) & = \Big\langle N_\ell A_{ff} N_\ell \mathbf{w}_\ell - N_\ell A_{fc} \mathbf{e}_\ell , \mathbf{v}\Big\rangle,
\end{align*}
with $\mathbf{v} = N_\ell \mathbf{v}$, resulting in the weak form
\begin{align*}
N_\ell A_{ff} N_\ell \mathbf{w}_\ell = N_\ell A_{fc} \mathbf{e}_\ell.
\end{align*}
\end{enumerate}
\end{proof}

\begin{remark}
A similar result to Lemma~\ref{th:cg} can be found in~\cite{Brannick:2007fb} in the Frobenius norm.
\end{remark}

Although $P_{\textnormal{ideal}}$ in~\eqref{eq:ideal} is motivated through (and optimal in the sense of~\eqref{eq:mu}) the WAP~\eqref{eq:wap}, and with an appropriate CF-splitting satisfies the SAP~\eqref{eq:sap}, $P_{\textnormal{ideal}}$ is not optimal in any sense with respect to the SAP\@. However, a similar derivation leads to an equivalent ``ideal interpolation'' operator with respect to the SAP, as introduced in Lemma~\ref{lem:ideal_sap}.

\begin{lemma}\label{lem:ideal_sap}
Let $A \in\mathbb{R}^{n \times n}$ be SPD and $P\in\mathbb{R}^{n\times n_c}$ take the form
$P = \begin{pmatrix}W \\ I\end{pmatrix}$, that is C-points are interpolated by injection. Then, consider
the SAP under the assumption that the pre-image of any vector $\mathbf{u}$ under $P$ is given by $\mathbf{u}_c$,
\begin{align*}
\hat\mu(P) := \max_{\mathbf{u}\neq 0} \frac{\|\mathbf{u} - P \mathbf{u}_c\|_A^2}{\|A\mathbf{u}\|^2}.
\end{align*}
Then for any smoothing scheme, $M$, $K_{MG} \leq 1 + \hat\mu(P)\frac{\|M\|}{\|A\|}$, and the optimal $P$ with respect to minimizing $\hat\mu$ is given by
\begin{align*}
\argmin_P \hat\mu(P) & = \begin{bmatrix} {(A_{ff}^2 + A_{fc} A_{cf})}^{-1}(A_{ff}A_{fc} + A_{fc}A_{cc}) \\ I\end{bmatrix},
\end{align*}
which is exactly ideal interpolation~\eqref{eq:ideal} for $A^2$.
\end{lemma}
\begin{proof}
First note from the SAP,
\begin{align*}
K_{MG} & \leq 1 + \left(\max_{\mathbf{u}\neq 0}\min_{\mathbf{w}_c} \frac{\|\mathbf{u} - P \mathbf{w}_c\|_A^2}{\|A\mathbf{u}\|^2}\right)\frac{\|M\|}{\|A\|} \\
& \leq 1 + \left(\max_{\mathbf{u}\neq 0} \frac{\|\mathbf{u} - P \mathbf{u}_c\|_A^2}{\|A\mathbf{u}\|^2}\right)\frac{\|M\|}{\|A\|}.
\end{align*}
Based on the proof of Theorem 3.1 and Corollary 3.2 in~\cite{FaVa2004}, it follows that
\begin{align*}
\argmin_P \max_{\mathbf{u}\neq 0} \frac{\|\mathbf{u} - P \mathbf{u}_c\|_A^2}{\|A\mathbf{u}\|^2} & =
	\argmin_P \max_{\mathbf{u}\neq 0} \frac{\|\mathbf{u} - P \mathbf{u}_c\|^2}{\|A\mathbf{u}\|^2} & =
	\argmin_P \max_{\mathbf{u}\neq 0} \frac{\|\mathbf{u} - P \mathbf{u}_c\|^2}{\|\mathbf{u}\|_{A^2}^2}.
\end{align*}
The final equation is $\mu(P)$~\eqref{eq:mu} as applied to $A^2$, and thus the
minimum is attained by ideal interpolation~\eqref{eq:ideal} as applied to $A^2$.
\end{proof}

\begin{remark}
An equivalent result holds for $A^*A$ and $\sqrt{A^*A}$ (as opposed
to $A^2$ and $A$), which is used in the following section on non-symmetric operators.
\end{remark}

\subsection{The non-symmetric case}\label{sec:theory:nonsymm}

The difficulties of convergence theory for non-symmetric $A$ lie in the fact that $A$ no longer
defines a norm, raising the question of an appropriate norm to measure convergence.
The spectral radius bounds asymptotic convergence factors~\cite{Notay:2010em,Wiesner:2014cy}; however,
consider the following error propagation matrix as an example:
\begin{align*}
E = \begin{pmatrix} \epsilon & k \\ 0 & \epsilon\end{pmatrix},
\end{align*}
for $\epsilon  <1$. Irrespective of $k$, $\rho(E) = \epsilon$, while
if $k=0$, then the convergence factor in any reasonable norm should be
$\epsilon$. But, if $k = 10^8$, let $\mathbf{e}_0 = {[0,1]}^T$ and the
convergence factor in the $l^2$-norm after one iteration is
\begin{align*}
\rho_1 = \frac{\|\mathbf{e}_1\|}{\|\mathbf{e}_0\|} = \|M\mathbf{e}_0\| = \sqrt{1000^2+\epsilon^2} \approx 1000.
\end{align*}
This demonstrates the downside of considering convergence for non-symmetric AMG methods
based on spectral radius, namely that iterations may diverge, and
the spectral radius bound on the convergence factor is not necessarily achieved in practice.
In the case of SPD $A$, the two-grid AMG error propagation operator, $E_{TG}$, is
$A$-symmetric and $\|E_{TG}\|_A = \rho(E_{TG})$.
In~\cite{BrMaMcRuSa2010}, energy-norm convergence is extended to non-symmetric
problems through $\sqrt{A^*A}$- and $\sqrt{AA^*}$-norms (using the square root to maintain the order of the problem).
Here, for non-singular $A$, energy minimization and \rnamg\ are related to two-grid theory in the $\sqrt{A^*A}$-norm.

Energy-minimization techniques for non-symmetric problems have also been proposed in the $A^*A$-norm
for $P$ and $AA^*$-norm for $R$, as opposed to the $A$-norm as referenced in Lemma~\ref{th:cg}
(see for example, fGMRES~\cite[(2.7)]{OlScTu2011} and CGNR~\cite[(2.34)]{OlScTu2011}), and
are the basis of forming transfer operators in \rnamg\ for non-symmetric problems.
Building on Lemmas~\ref{th:wap_sap} and~\ref{th:cg} gives the following result in Lemma~\ref{lem:nonsymm}.
Coupled with Conjecture~\ref{conj_stab}, two-grid convergence follows from~\cite{BrMaMcRuSa2010}.
Lemma~\ref{lem:nonsymm} provides a meaningful theoretical motivation
for energy minimization as applied to non-symmetric problems.
\begin{lemma}\label{lem:nonsymm}
The solution to energy-minimization in the $A^*A$- and $AA^*$-norms satisfy the non-symmetric strong
approximation property in the $\sqrt{A^*A}$- and $\sqrt{AA^*}$-norms, respectively, that is, for all $\mathbf{v}$
there exists a $\mathbf{v}_{c_1},\mathbf{v}_{c_2}$ such that
\begin{align*}
\left\|\mathbf{v} - P_{\textnormal{ideal}}^{A^*A} \mathbf{v}_{c_1}\right\|_{\sqrt{A^*A}}^2 & \leq \frac{K}{\|A^*A\|} \|\mathbf{v}\|_{A^*A}^2, \\
\left\|\mathbf{v} - P_{\textnormal{ideal}}^{AA^*} \mathbf{v}_{c_2}\right\|_{\sqrt{AA^*}}^2 & \leq \frac{K}{\|AA^*\|} \|\mathbf{v}\|_{AA^*}^2.
\end{align*}
\end{lemma}
\begin{proof}
The proof follows immediately from the equivalence of the WAP$(A^2)$ and SAP$(A)$ (Lemma~\ref{th:wap_sap})
and the convergence of energy-minimization to $P_{\textnormal{ideal}}$, in this case $P_{\textnormal{ideal}}$ for $A^*A$ and $AA^*$ (Lemma~\ref{th:cg}).
\end{proof}
\begin{conjecture}[Stability]\label{conj_stab}
Let $A$ be nonsingular. The non-orthogonal coarse-grid correction given by transfer operators
$R = {(P_{\textnormal{ideal}}^{AA^*})}^T$ and $P = P_{\textnormal{ideal}}^{A^*A}$ is \textit{stable}, that is
\begin{align} \label{eq:stability}
\Big\| P(RAP)^{-1}RA\Big\|_{\sqrt{A^*A}} = \Big\|I - P(RAP)^{-1}RA\Big\|_{\sqrt{A^*A}} = C,  \end{align}
for some constant $C$, independent of mesh spacing.
\end{conjecture}

\begin{remark}
In the symmetric case, $R = P^T$ and $C = 1$, as the coarse-grid correction is an $A$-orthogonal
projection. In the non-symmetric case, the stability assumption necessary for two-grid convergence
as shown in~\cite{BrMaMcRuSa2010} is
primarily to ensure a non-singular and reasonably conditioned coarse-grid operator, $RAP$.
Conjecture~\ref{conj_stab} appears to hold in general, but expanding the ideal operators and forming $RAP$ or
the full projection~\eqref{eq:stability} does not provide a clear method to bound its norm.
However, in practice the root-node approach over traditional aggregation offers greater stability
of the non-orthogonal projection through enforcing the identity over C-points in transfer operators.
Specifically, when forming the coarse grid, the identity block in $R$ and $P$ help ensure the
non-singularity of $RAP$.
\end{remark}

\section{Numerical results}\label{sec:numerical}

In the following, the open source software package PyAMG~\cite{BeOlSc2008}
is used to implement \rnamg\footnote{See \code{root\_node\_solver}} and \saamg\footnote{See \code{smoothed\_aggregation\_solver}}.
In each routine, cumulative estimates of the SC and CC (see Section~\ref{sec:cost_overview})
are tracked in WUs.
All costs in the setup phase are included, such as computing the triple matrix product $R A P$.
In addition, the CC estimates include pre- and post-smoothing, computing and
restricting the residual, and coarse-grid correction, but the coarsest-grid direct solve is not included.
Measured CC multiplied by the time to perform one SpMV is highly accurate when compared with
the wall-clock time of the solve phase. Although measured SC does not track as closely with wall-clock time
(likely due to a more complicated setting with memory allocation, conditionals, etc), they still provide a good
estimate of setup cost, and differ from wall-clock times by a small constant. 

As with many AMG methods, \rnamg\ has several parameters.
The optimal SOC measure is problem dependent, although the evolution measure
tends to be more robust for problems with strong anisotropy. Still, the
evolution measure does result in higher setup costs, as observed in
Section~\ref{sec:numerical:3DAni}. Another parameter is that of
the degree $d$ of sparsity pattern for $P$ (see
line~\ref{line:sparsity} in Algorithm~\ref{alg:rn_setup}). Generally, more
difficult problems such as anisotropic diffusion require $d$ to be as large as 4--5.
The number of smoothing iterations applied to $P$ is generally set to $\lceil
1.5d\rceil$, as a larger sparsity pattern requires more iterations to minimize
the energy of the columns.  For SPD problems, the CG-based variant of
energy minimization is the least costly, while GMRES is used for non-symmetric
problems (GMRES does not offer improved convergence for SPD problems; see Remark~\ref{rem:cg}).
Filtering $P$ also requires user-level decisions
(lines~\ref{line:prefilter}~and~\ref{line:postfilter} in Algorithm~\ref{alg:rn_setup}).
Based on experimentation, a drop tolerance of $\theta$ is most effective (in contrast to
a fixed number of elements per row $k$) and is used in the following tests.
Optimal values depend on the connectivity of matrix and are beyond the scope of the numerical
tests presented, however $\theta\in[0.05,0.25]$ is used in most scenarios.

In each setup phase below, a maximum coarse-grid size of 20 is used, and
candidate vectors are improved (line~\ref{Bimproveline},
Algorithm~\ref{alg:rn_setup}) with four sweeps of relaxation.  For cycling, each
test uses an accelerated V-cycle to iterate to a relative residual tolerance of
$10^{-8}$.  Details of the specific relaxation scheme, acceleration type, and
other details are specified on a problem-by-problem basis below.

\subsection{3D-diffusion}\label{sec:numerical:3DAni}

\subsubsection{Filtering and complexity}

The first example demonstrates the effectiveness of the proposed filtering
strategy in reducing the cycle and setup complexity of a RN solver. Although
filtering is a key component of \rnamg\ on nearly all problems, it is especially
applicable in 3D, where there is high connectivity between nodes, resulting
in relatively dense operators. Consider the anisotropic diffusion problem
\begin{align}\label{eq:3Dani}
  u_{xx} + u_{yy} + 0.001 u_{zz} = f.
\end{align}
Linear finite elements are used to discretize~\eqref{eq:3Dani} on an
unstructured tetrahedral mesh of the unit cube, with homogeneous, Dirichlet
boundaries, yielding a matrix with approximately 2.65M DOFs. While the anisotropy is aligned with the coordinate axis, the
unstructured mesh yields a variety of
non-grid-aligned anisotropies, known to be more difficult for AMG than the
grid-aligned case (see Section~\ref{sec:numerical:TAD}). Table~\ref{tab:3DAni}
shows complexities and average convergence factors ($\rho$) for solving~\eqref{eq:3Dani}
using various combinations of pre- and post-filtering on $P$.
A V$(1,1)$-cycle with
symmetric Gauss-Seidel relaxation and CG acceleration is used. The evolution SOC
measure is used, with a drop tolerance of 4.0,
and CG energy minimization with $d=4$.
{\renewcommand{\tabcolsep}{4pt}
\begin{table}[!ht]
  \centering
  \begin{tabular}{c S SS SS SS}    \toprule
Pre-filter $\theta$  & {--}   & {$0.1$}& {$0.2$} & {--}    & {--}    & {$0.1$} & {$0.2$} \\
Post-filter $\theta$ & {--}   & {--}   & {--}    & {$0.1$} & {$0.2$} & {$0.1$} & {$0.2$} \\
     \midrule
     SC              & 1157.9 & 161.7  & 123.2   & 581.5   & 432.9   & 156.2   & 129.8 \\
     OC              &  3.0   & 1.4    & 1.3     & 1.8     & 1.5     & 1.4     & 1.3 \\
     CC              & 18.9   & 8.0    & 7.1     & 10.1    & 8.0     & 7.7     & 6.9 \\
     \midrule
     $\rho$              & {0.52} & {0.60} & {0.62}  & {0.51}  & {0.53}  & {0.59}  & {0.61} \\
  \bottomrule
  \end{tabular}
  \caption{Impact of filtering for the 3D-anisotropic diffusion problem.}\label{tab:3DAni}
\end{table}
}

Table~\ref{tab:3DAni} shows up to an order-of-magnitude reduction in SC
and more than a 50\% reduction in CC by using filtering.  While
filtering does not guarantee a reduction in complexity, significant savings
are often observed with only a marginal impact on convergence.
In some situations, filtering has been observed to not only reduce cost,
but also improve convergence (see Section~\ref{sec:transport}).

The SC in Table~\ref{tab:3DAni} is broken down into four main
categories in Table~\ref{tab:sc}. ``Aggregation'' is the cost of
computing the strength matrix $S$ and forming aggregates $\mathcal{A}$, most of
which is due to using the evolution measure.
``Candidates'' refers to relaxing candidate set $B$
and restricting $B$ to a coarse level.
Column ``$P$'' refers to the cost of forming the tentative
interpolation operator and applying energy-minimization smoothing iterations to construct $P$, while
column ``$RAP$'' represents a measure of the triple-matrix product.
Each column gives the total cost for the given processes over all levels, measured in WUs.
{\renewcommand{\tabcolsep}{4pt}
\begin{table}[!ht]
  \centering
  \begin{tabular}{cc SSSS S}    \toprule
Pre-filter $\theta$    & Post-filter $\theta$    & {Aggregation} & {Candidates} & {$P$} & {$RAP$}   & {Total SC} \\
\midrule
{--} & {--} & 478.7 & 24.1 & 469.7 & 271.1 & 1243.6 \\
0.2  & {--} & 48.5  & 10.4 & 55.1  & 9.2   & 123.2 \\
{--} & 0.2  & 64.1  & 11.7 & 338.7 & 18.5  & 432.9 \\
0.2  & 0.2  & 46.3  & 10.1 & 65.4  & 7.9   & 129.8 \\
  \bottomrule
  \end{tabular}
  \caption{Break down of setup cost in WUs for 3D-anisotropic diffusion.}\label{tab:sc}
\end{table}
}

Filtering $P$ has a direct impact on all setup components
that use $R$ and $P$ matrix operations (see ``$RAP$'' and ``Candidates'' in Table~\ref{tab:sc}).
Consequently, this reduces the cost of restricting the residual and the coarse-grid correction
as measured in the CC\@, along with the cost of smoothing $P$,
which is the focus of pre-filtering. Table~\ref{tab:sc} also highlights the
high cost of the strength measure in cases when coarse-grid complexity is not
contained~---~i.e.\ filtering is not used, demonstrating the additional benefit of
reduced costs on all subsequent grids through a sparser coarse-grid operator, $A_c = RAP$.

\begin{remark}\label{rem:cg}
Although GMRES energy-minimization targets the SAP and CG the WAP,
when applied to SPD problems, there is not a notable difference in convergence.
For instance if GMRES is used in Table~\ref{tab:3DAni},
then the convergence rates change by no more than
0.005 and operator complexities remain essentially the same.
\end{remark}

\subsubsection{Modified evolution SOC}\label{sec:numerical:3DAni:soc}

One contribution of this paper are the detailed SC estimates.
Precise SC estimates are important when comparing AMG methods~---~e.g., components such as energy-minimization and the evolution SOC measure may add significantly to the overall setup costs.
Moreover, these estimates help identify
areas the are opportune for cost reduction.

In considering Table~\ref{tab:sc}, one possible area for cost reduction is the
aggregation phase, where the cost of the SOC computation dominates.
Traditional SOC measures require only a few work units, usually 2 or 3 times the
operator complexity.  However, Table~\ref{tab:sc} shows that evolution-based SOC is a substantial part of the setup phase, which is attributed to the global spectral radius estimate used in weighted Jacobi~\cite{OlScTu2009}.
This estimate is calculated with an Arnoldi/Lanczos
process and costs roughly 15 matrix-vector multiplies on each level.  However,
alternative methods~\cite{BaFaKoYa2011} use $\ell_1$-Jacobi relaxation
to provide an inexpensive local row-wise weight.
This alternative is explored here for the evolution SOC measure.

Table~\ref{tab:sc_cheap} depicts detailed SC results for using the modified
$\ell_1$-Jacobi evolution measure (cf. Table~\ref{tab:sc}). This change
results in similar operator and cycle complexities
and nearly identical convergence rates as the original evolution measure.
When comparing Tables~\ref{tab:sc}~and~\ref{tab:sc_cheap}, it is apparent that
the ``Aggregation'' phase has been significantly reduced in cost.  Additionally, when
examining the most efficient solvers, where pre-filtering uses
$\theta = 0.2$, the overall savings in the setup complexity are roughly 20\%.
{\renewcommand{\tabcolsep}{4pt}
\begin{table}[!ht]
  \centering
  \begin{tabular}{cc SSSS S}    \toprule
Pre-filter $\theta$    & Post-filter $\theta$    & {Aggregation} & {Candidates} & {$P$} & {$RAP$}   & {Total SC} \\
\midrule
{--} & {--} & 449.1 & 24.4 & 473.0 & 249.5 & 1195.9 \\
0.2  & {--} & 25.5  & 10.4 & 53.6  & 7.6   & 97.1  \\
{--} & 0.2  & 43.2  & 11.7 & 326.7 & 16.9  & 398.6 \\
0.2  & 0.2  & 22.7  & 10.2 & 63.3  & 6.9   & 103.1 \\
  \bottomrule
  \end{tabular}
  \caption{Break down of setup cost in WUs for 3D-anisotropic diffusion.}\label{tab:sc_cheap}
\end{table}
}

\subsection{Totally anisotropic diffusion}\label{sec:numerical:TAD}

Diffusion-like operators are prototypical AMG problems, as they are elliptic and
SPD\@. However, with strong anisotropy, these problems still pose a challenge to
multilevel solvers. A 3D example is used in Section~\ref{sec:numerical:3DAni} to demonstrate
filtering. In this example, consider a 2D rotated anisotropic diffusion problem on the domain $\Omega$ of the form
\begin{align*}
-\nabla\cdot Q^T D Q\nabla u & = f\quad\textnormal{for ${\Omega} = [0,1]^2$},\\
                           u & = 0\quad\textnormal{on $\partial\Omega$},
\end{align*}
where
\begin{equation*}
Q = \begin{bmatrix} \cos{\psi} &          - \sin{\psi} \\
                    \sin{\psi} & \phantom{-}\cos{\psi}
    \end{bmatrix},
\qquad
D = \begin{bmatrix}1 & 0 \\ 0 & \epsilon \end{bmatrix}.
\end{equation*}
Here, $\epsilon$ represents anisotropy and $\psi$ the angle of rotation from the
coordinate axis, which both contribute to the difficulty of this
problem~\cite{OlScTu2011,OlScTu2009,Gee:2009dy,DAmbra:2013iwa,BrBrKaLi2015,Notay:2010um,Brandt:2011tu,Brannick:2012bl,BrBrMaMaMcRu2006,DeFaNoYa_2008,MaBrVa1999,Sc2012,Chen:2015vx,Brannick:2010hz, BrBrKaLi2015}.
To see this, consider the non-rotated case of $\psi=0$, which has a spectrum of the form
\begin{align*}
-\nabla\cdot D\nabla u_{jk} & =  \lambda_{jk}u_{jk}, \hspace{3ex}\textnormal{where} \\
 u_{jk} & =  \sin(j\pi x)\sin(k\pi y), \\
 \lambda_{jk} & = \pi^2(j^2 + \epsilon k^2),
\end{align*}
for $j,k\in\mathbb{Z}^+$.
If $\epsilon = 1$, then the lowest energy mode is the lowest Fourier mode:
$u_{11} = \sin(\pi x)\sin(\pi y)$, which is locally representative of all low-energy modes.
However, if $\epsilon\approx 0$, for small $j$ there are high-frequency
eigenfunctions in the $y$-direction $(k \gg 0)$ with relatively small
eigenvalues, that are no longer represented locally by the lowest Fourier mode. As a result, relaxation schemes
are unable to capture these modes, while coarse-grid correction
is not equipped to handle such \textit{hidden} high frequency error.

For angles $\psi$ aligned with the mesh, line relaxation or semi-coarsening
along the direction of anisotropy can be used~\cite{Schaffer:1998ui}. However, for
strong anisotropies, $\epsilon\approx 0$, with angles that are not aligned with
the mesh, efficient and effective multigrid solvers remain elusive.  This
section considers a finite element discretization of strongly and \textit{totally}
anisotropic diffusion, $\epsilon = 0.001$ and $\epsilon = 0$, respectively, on a unit square
with Dirichlet boundary conditions. Totally anisotropic diffusion is
particularly challenging as the problem is effectively reduced to a sequence of 1D
problems on a 2D domain.

\begin{figure}[ht!]
  \centering
  \begin{subfigure}[b]{0.48\textwidth}
    \includegraphics[width=\textwidth]{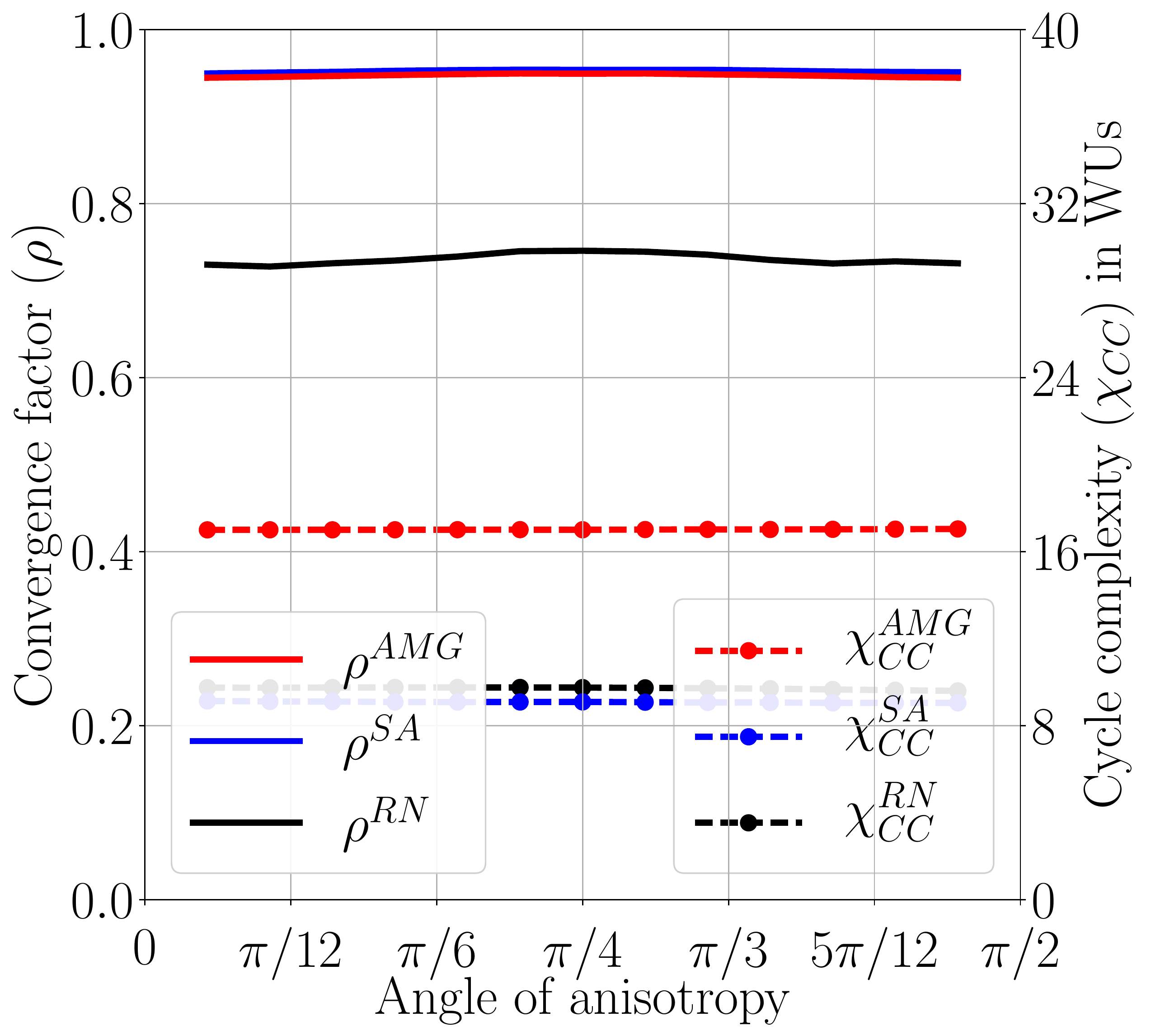}
    \caption{Unstructured mesh: $n=5\times 10^6$.}\label{fig:TAD-unstructured}
  \end{subfigure}
  \hfill
  \begin{subfigure}[b]{0.48\textwidth}
    \includegraphics[width=\textwidth]{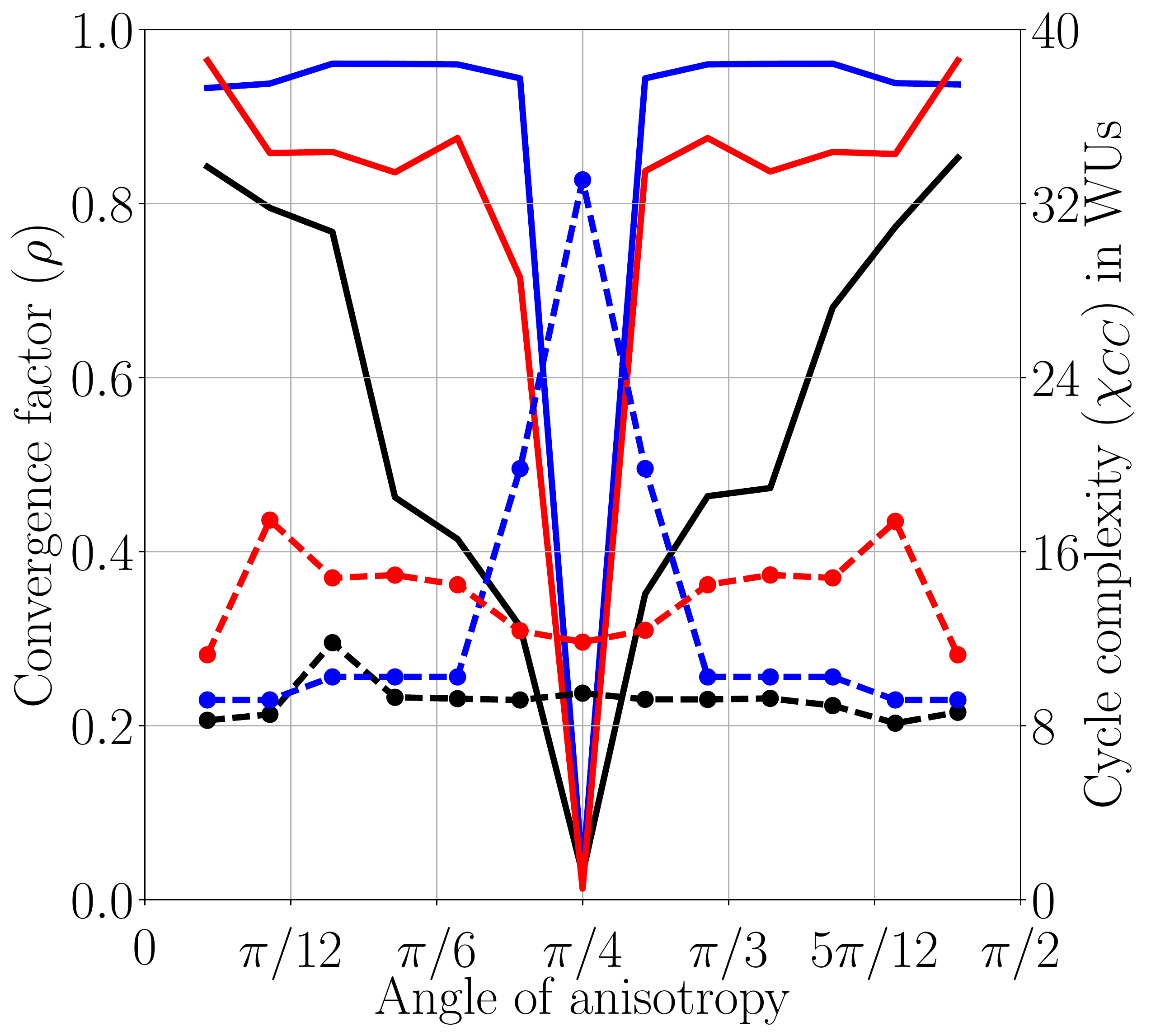}
    \caption{Structured mesh: $n=8\times 10^6$.}\label{fig:TAD-structured}
  \end{subfigure}
  \caption{Results for totally anisotropic diffusion ($\epsilon=0$), with angles in $(0,\sfrac{\pi}{2})$.  A structured grid of size $2000\times 2000$ and an unstructured mesh with resolution $h\approx 1/2000$ are used. Legend entries in (a) apply to (b) as well. \rnamg, shown in black, outperforms \saamg and \cfamg 
  in all cases.
    }\label{fig:TAD}
\end{figure}

Figure~\ref{fig:TAD} compares \rnamg, \cfamg, and \saamg\ applied to totally anisotropic diffusion for various angles.
All solves use a symmetric V{(1,1)}-cycle with symmetric Gauss-Seidel relaxation and CG acceleration.
Grid-aligned anisotropies are generally easier to solve than non-grid aligned, hence the excellent convergence factors
at $\theta = \sfrac{\pi}{4}$ on the structured mesh. In the case of an
unstructured mesh, all angles are effectively non-grid aligned resulting in consistent performance across angles.

Classical AMG uses a classical strength measure with drop tolerance of $0.5$, and standard CF-splitting and interpolation \cite{RuStu1987}. 
Smoothed aggregation uses a symmetric strength measure (with a drop tolerance of $0.0$~---~i.e.\ the strength matrix
is given by $A$ and normalized as in~\eqref{eq:soc}) and two iterations of weighted Jacobi interpolation smoothing \cite{VaMaBr1996}.
The root-node solver in this case uses two steps of an evolution strength measure (with a drop tolerance of $4.0$), along with six
iterations of CG energy-minimization smoothing of $P$. For the energy minimization, a degree $d=4$ sparsity pattern is used with
filtering, $\theta_{\text{pre}} = \theta_{\text{post}} =0.1$. Two and three Jacobi smoothing iterations were applied to the \saamg solver
in an attempt to mimic the expanded sparsity pattern used in \rnamg, but the convergence with respect to CC did not improve. 

Figure~\ref{fig:TAD} highlights the effectiveness of \rnamg\ over all angles with moderate operator and cycle complexities.
For time-to-solution with respect to floating point operations and wall-clock time, \rnamg achieves between 3--30$\times$ speed-up in
comparison to \saamg and \cfamg on a structured mesh and unstructured mesh, with moderate cycle complexities in all cases. It should
be noted that performance of \saamg improves when using the modified evolution SOC measure introduced in Section
\ref{sec:numerical:3DAni:soc}, but \rnamg still performs 2--3$\times$ faster with respect to time and complexity. This is a notable
achievement in performance for this problem, as anisotropic diffusion remains a significant challenge to most solvers.

\begin{figure}[!ht]
  \centering
  \begin{subfigure}[b]{0.48\textwidth}
    \includegraphics[width=\textwidth]{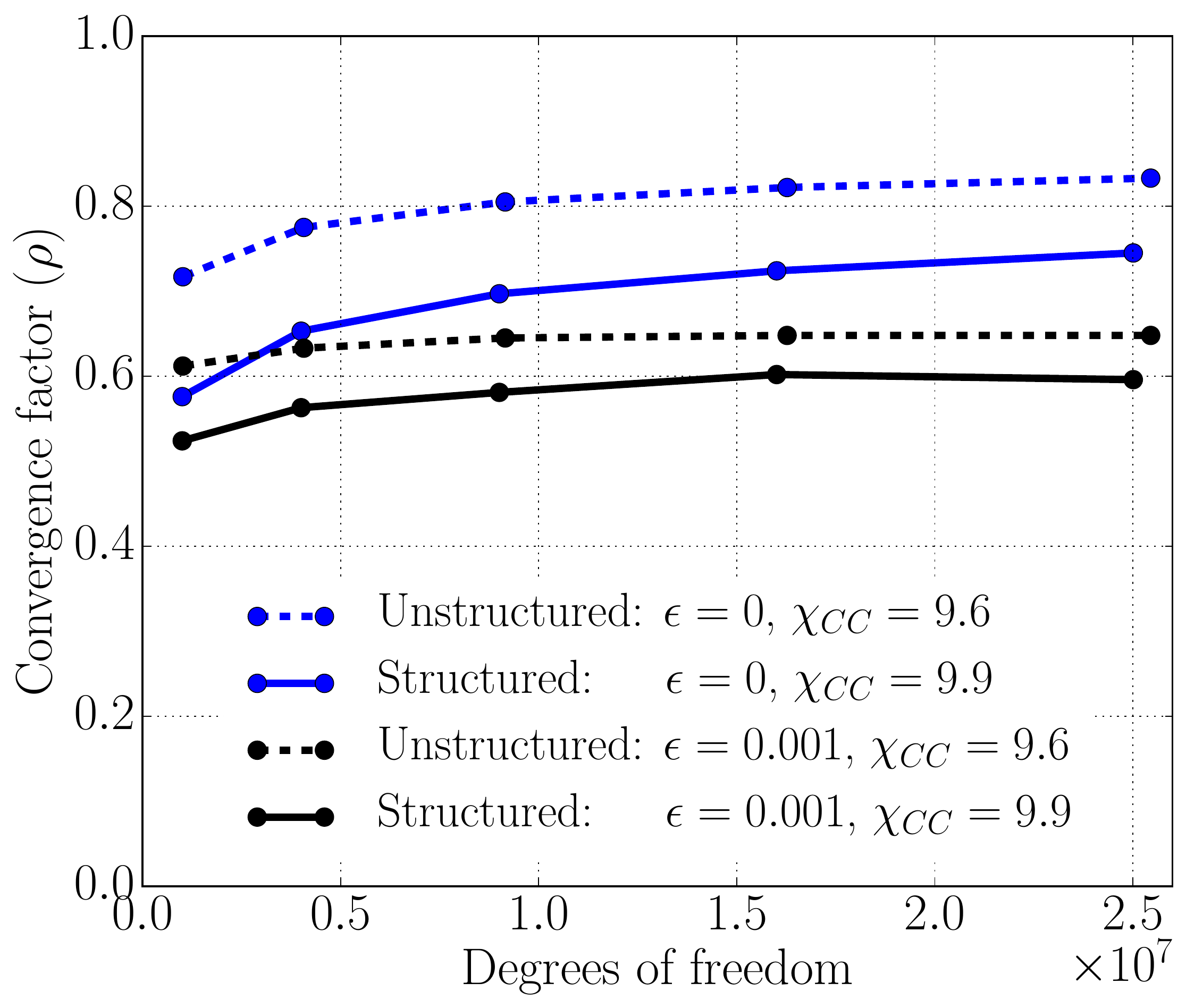}
    \caption{V$(1,1)$-cycle.}\label{fig:TAD_scale-V}
  \end{subfigure}
  \hfill
  \begin{subfigure}[b]{0.48\textwidth}
    \includegraphics[width=\textwidth]{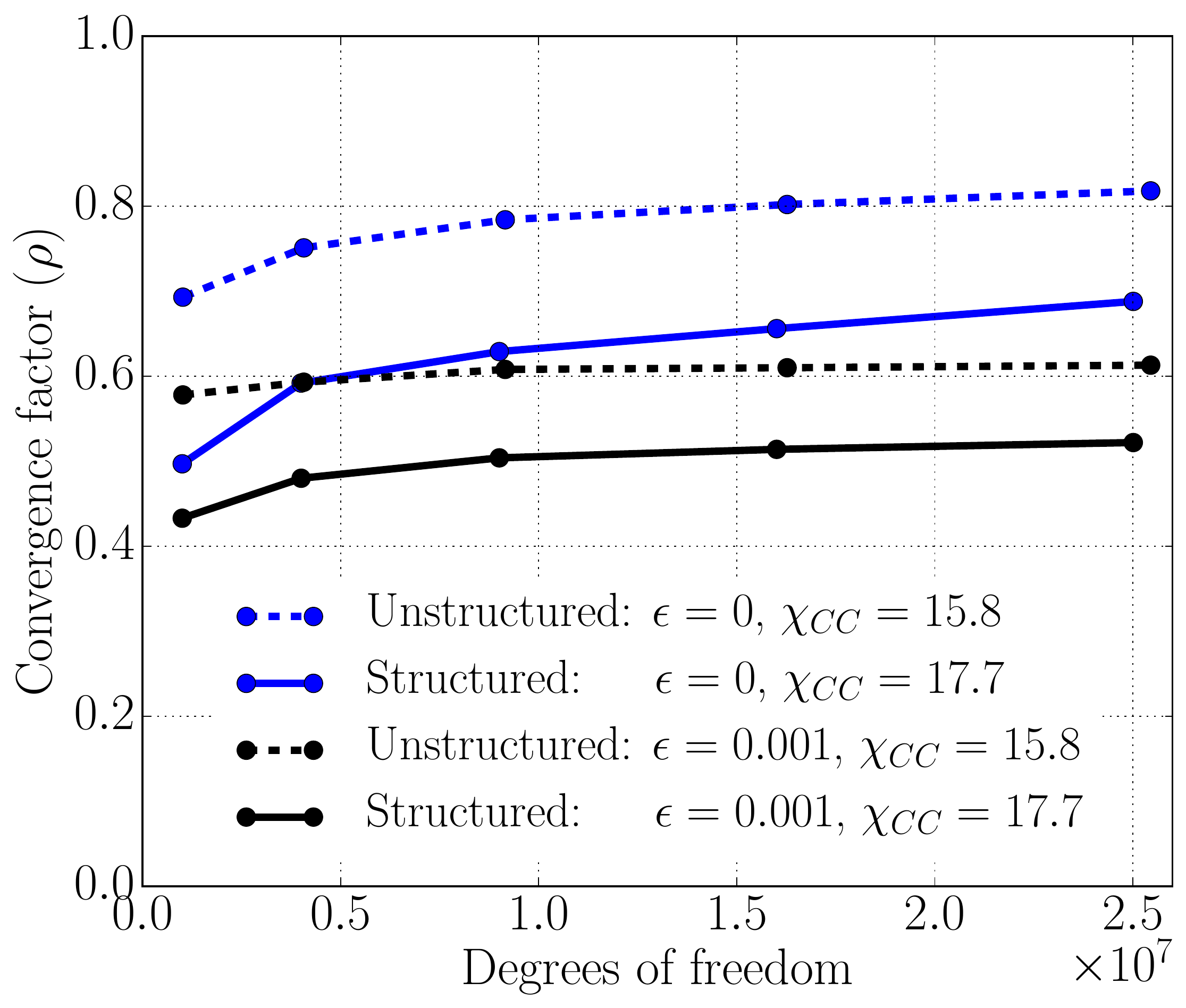}
    \caption{W$(1,1)$-cycle.}\label{fig:TAD_scale-W}
  \end{subfigure}
  \caption{Scaling results for \rnamg\ for anisotropic diffusion with $\epsilon = 0.001, \epsilon = 0$, and $\theta = 3\pi/16$. Cycle complexity, $\chi_{\textnormal{CC}}$, is shown and is constant for all problem sizes.}
\label{fig:TAD_scale}
\end{figure}

Figure~\ref{fig:TAD_scale} demonstrates the scaling of convergence factors as
problem size increases for $V(1,1)$- and $W(1,1)$-cycles. In the case of $\epsilon
= 0.001$, V-cycle convergence factors asymptote and scale perfectly, independent
of $h$, on structured and unstructured meshes, up to 25 million unknowns; the
SC and CC also scale, but are not shown. However, in the case of $\epsilon = 0$ there is
a slow growth
in the convergence factor as the problem size increases for both V- and
W-cycles.
To analyze, consider a convergence factor,
$\rho(h)$, dependent on spatial step size $h$:
\begin{align}\label{eq:asymp}
\rho = \bar{\rho} (1-a h^q),
\end{align}
where $q=1$ in the case of linear finite elements, $a$ is some constant, and
$\bar{\rho}$ is the asymptotic convergence factor, $\lim_{h\to 0}\rho =
\bar{\rho}$. A log of~\eqref{eq:asymp} and an expansion yields $\log(1-ah) = -ah +
O(h^2)$ or $-\log(\rho) = -\log(c) + ah$. A linear fit on the three smallest
step sizes in Figure~\ref{fig:TAD_scale} for $\epsilon = 0$ leads to the following asymptotic
convergence factors
\begin{align*}
\textnormal{structured:} \quad   \rho \rightarrow 0.82, \quad
\textnormal{and unstructured:} \quad \rho \rightarrow 0.88.
\end{align*}
Although an asymptotic convergence factor of approximately $0.88$ is relatively slow,
\textit{scalable} convergence of totally anisotropic diffusion on unstructured meshes has not been
achieved by other AMG methods.

In practice, using W-cycles over V-cycles should increase the accuracy
of the coarse-grid correction. However, Figure~\ref{fig:TAD_scale} reveals W-cycles
offer only minor improvements. This indicates that the algebraically smooth error is not well represented
by the coarse grid, and that improved strength measures and coarsening routines~\cite{OlScTu2009,DAmbra:2013iwa,
BrBrKaLi2015,Notay:2010um,Brannick:2012bl,BrBrMaMaMcRu2006} may
improve the coarse grid, thereby improving convergence for this problem.

\subsection{Recirculating flow (non-symmetric)}\label{sec:numerical:flow}

One of the benefits of the root-node approach is the ability to handle a variety of
problems, including non-symmetric problems and systems, without redesign of
methodology and implementation. In this section, a standard recirculating
flow example known as the \emph{double-glazing problem}
is used~\cite{ElSiWa2014}, which models temperature distribution
over a domain when an external wall is hot. The governing PDE is given by
\begin{equation}\label{eq:flow}
-\epsilon \nabla \cdot \nabla u + \mathbf{b}(\textbf{x}) \cdot \nabla u = f,
\end{equation}
where $\epsilon = 0.005$ and wind is given by $\mathbf{b}(\textbf{x}) = [2x_1(1-x_0^2), -2x_0(1-x_1^2)]$.
Dirichlet boundaries are imposed on the domain $[0,1] \times [0,1]$, where
$u=0$ on the north, south and west sides of the domain, and $u = 1$
on the east, leading to boundary layers near corners with discontinuities.

A standard Galerkin finite element method (GFEM) based on a regular triangular mesh
is used, resulting in a non-symmetric discrete linear system.
Multigrid theory for non-symmetric problems is less
developed in comparison to the symmetric case; however,
\saamg\ has been extended to
non-symmetric problems~\cite{Sala:2008cv} and is used in this section (see Algorithm~\ref{alg:rn_setup}).

Each example uses a V$(1,1)$-cycle of weighted Jacobi with GMRES acceleration.  While
Jacobi relaxation is not guaranteed to converge for a non-symmetric problem, it remains around half
the cost of using relaxation on the normal equations.
In Section~\ref{sec:transport} an upwind discretization is considered which
requires relaxation on the normal equations for effective convergence. For \saamg,
a classical strength measure (with drop tolerance $0.25$) is used along with one and three steps of Jacobi smoothing
steps applied to $P$, labeled $\textnormal{SA}_1$ and $\textnormal{SA}_3$, respectively. The
symmetric traditional SA strength measure is not used due to the non-symmetry of the problem. Two steps of the evolution measure
(with drop tolerance $ 3.0$) is used for \rnamg\ along with two iterations of GMRES energy minimization
for $P$ with $d=1$ (labeled $\textnormal{RN}_1$), and five iterations of GMRES
energy minimization for $P$ with $d=3$ (labeled $\textnormal{RN}_3$).

Table~\ref{tab:recircgfem} demonstrates optimal results are achieved
for this example with no filtering and a small sparsity pattern ($d=1$).
This agrees with practical experience: generally it is effective to
increase the filtering tolerance as the degree of the sparsity pattern for $P$ increases,
or as the connectivity of matrix $A$ increases. This is observed in the
3D-anisotropic diffusion problem, where high connectivity and a $d=4$ sparsity pattern allows for a large $\theta = 0.2$.
If the sparsity pattern increases in distance
from the root-node, or the matrix is highly connected, then it is likely there are
entries that are not critical to performance and are candidates for removal.
{\renewcommand{\tabcolsep}{4.5pt}
\begin{table}[!ht]
  \centering
  \begin{tabular}{c c | cccc | cccc | cccc }   \toprule
 & & \multicolumn{4}{c |}{$2000\times2000$} & \multicolumn{4}{c |}{$3000\times3000$} & \multicolumn{4}{c }{$4000\times4000$} \\
$\theta$ & $d$   & SC & OC & CC & $\rho$ & SC & OC & CC & $\rho$ & SC & OC & CC & $\rho$   \\ \midrule
\multirow{2}{*}{--} & $\textnormal{SA}_1$ & 72 & 1.4 & 5.2 & 0.74 & 71 & 1.4 & 5.2 & 0.82 & 71 & 1.4 & 5.2 & 0.88 \\
		& $\textnormal{SA}_3$ & 229 & 2.3 & 11.2 & 0.96 & 230 & 2.3 & 11.2 & 0.96 & 227 & 2.3 & 11.2 & 0.93 \\ \midrule
\multirow{2}{*}{--}& $\textnormal{RN}_1$ & 98 & 1.4 & 5.1 & 0.46 & 96 & 1.4 & 5.0 & 0.50 & 95 & 1.4 & 4.9 & 0.45 \\
		& $\textnormal{RN}_3$ & 403 & 2.5 & 9.7 & 0.68 & 407 & 2.4 & 9.4 & 0.63 & 405 & 2.3 & 9.2 & 0.76 \\\midrule
\multirow{2}{*}{0.1} & $\textnormal{RN}_1$ & 126 & 1.4 & 5.1 & 0.52 & 125 & 1.4 & 5.0 & 0.56 & 124 & 1.4 & 4.9 & 0.56 \\
		& $\textnormal{RN}_3$ & 284 & 1.8 & 6.7 & 0.53 & 287 & 1.8 & 6.7 & 0.52 & 285 & 1.8 & 6.6 & 0.63 \\ \bottomrule
  \end{tabular}
\caption{Non-symmetric \saamg\ and \rnamg\ for the recirculating flow problem.}\label{tab:recircgfem}
\end{table}
}

Classical CF AMG is not designed for non-symmetric problems, making RN AMG
the clear choice for a problem such as this. Root-node AMG achieves more than
a $6\times$ speed-up over SA AMG for the largest problem size considered,
\textit{with a lower CC}, and only slightly larger SC\@. Furthermore, RN AMG convergence factors
with degree $d=1$ appear to have reached an asymptote, while SA AMG is still demonstrating
a steady increase with problem size.

\begin{remark}
With the recirculating flow, as the
grid-size ($h$) approaches zero, there are two competing factors that contribute to the numerical
difficulty of the problem.
As $h\to0$, the diffusive part of the problem, $-\epsilon\nabla\cdot\nabla$,
becomes increasingly dominant, as the discretization scales like $\frac{1}{h^2}$, while
$\mathbf{b}(\mathbf{x})\cdot \nabla$ scales like $\frac{1}{h}$. The resulting linear system
is more symmetric and diffusion-like, which is preferable for AMG\@.
However, as $h\to 0$, convergence factors often increase to an asymptotic value
(see~\eqref{eq:asymp} and Figures~\ref{fig:TAD_scale-V}~and~\ref{fig:TAD_scale-W}), due to smaller eigenvalues and an
increasing number of levels in the AMG hierarchy. Together, these factors correspond to the so-called \textit{half-grid Reynolds number} or \textit{cell Reynolds number}, $R_h = \frac{|\mathbf{b}|h}{2\epsilon}$,
where convergence factors are expected to be consistent for a fixed $R_h$, and degrade for
$R_h \gg 1$. Figure~\ref{fig:halfgrid} demonstrates this phenomenon.
Convergence factors asymptote for each half-grid Reynolds number
as $h\to 0$. In this case, convergence degrades by a factor of 10 when increasing from $R_h=1.25$ to $R_h = 2.5$; thus it is faster to solve a refined problem with $R_h = 1.25$ (and several times as many
DOFs), rather than a system with $R_h = 2.5$.
\end{remark}
\begin{figure}[!t]
\begin{center}
\includegraphics[width=0.9\textwidth]{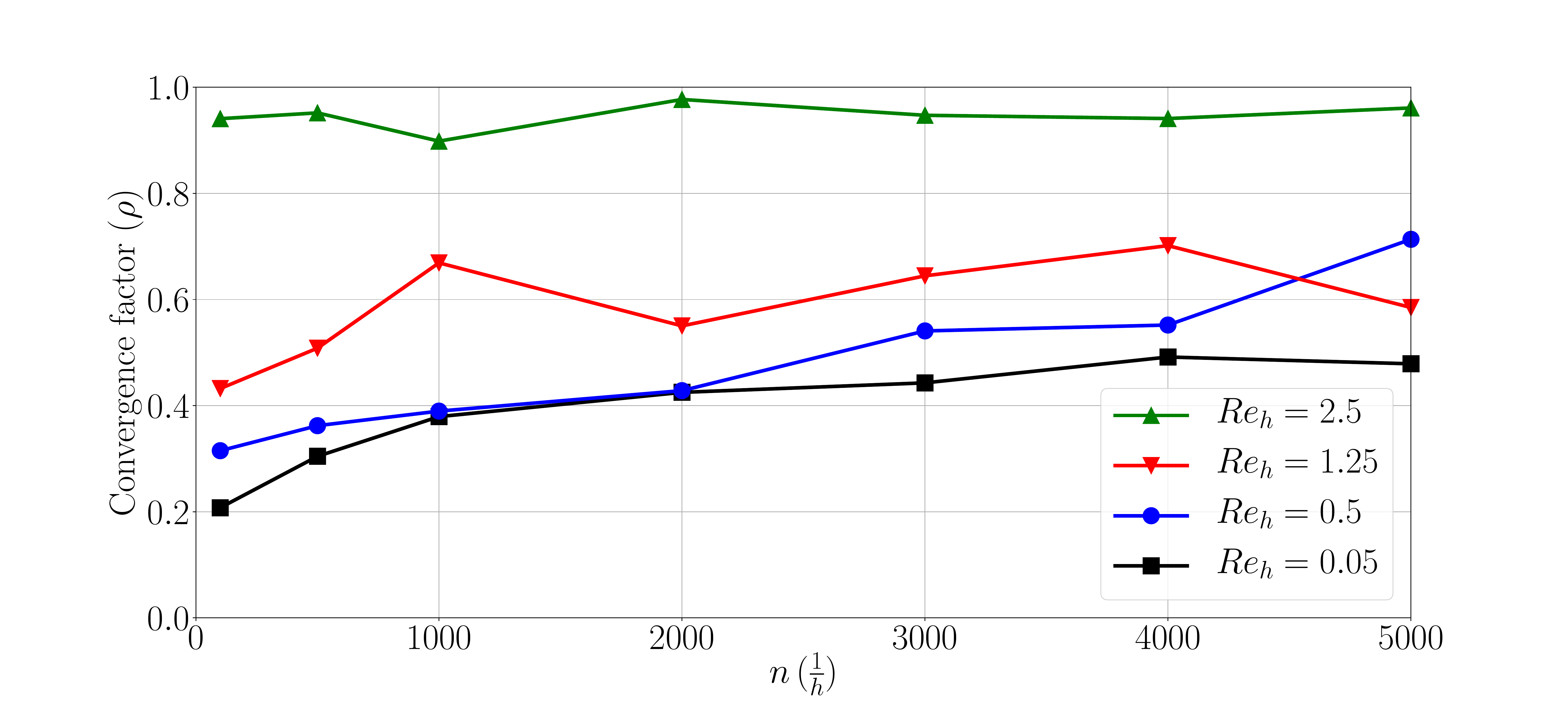}
\caption{\rnamg\ convergence factor as a function of $h$, for fixed half-grid Reynolds numbers,
$R_h \in\{0.05, 0.5,1.25, 2.5\}$.}
\label{fig:halfgrid}
\end{center}
\end{figure}

Based on the scalability of \rnamg\ performance for fixed $R_h$, with the appropriate grid-size,
most convection-diffusion problems as in~\eqref{eq:flow} should exhibit scalability in convergence.
Of course, this relies on the diffusion operator dominating the discrete system. As a limiting case,
\rnamg\ is applied to the steady-state transport equation in the following section.

\subsection{Upwind transport (non-symmetric)}\label{sec:transport}

In this section, a highly non-symmetric upwind discretization of the steady-state
transport equation is considered. Define the domain as $\Omega = (0,1)\times(0,1)$,
with constrained inflow boundaries $\Gamma_{\textnormal{in}} = \{ (x,y) \,:\, x = 0 \text{ or } y = 0\}$,
and free outflow boundaries $\Gamma_{\textnormal{out}} = \{ (x,y) \,:\, x = 1 \text{ or } y = 1\}$.
The steady-state transport equation is then given as
\begin{align}\label{eq:transport}
\begin{split}
\mathbf{b}(x,y) \cdot\nabla u + c(x,y)u & = f(x,y) \hspace{3ex}\Omega, \\
u & = g(x,y) \hspace{3ex}\Gamma_{\textnormal{in}}.
\end{split}
\end{align}
Root-node AMG performance is demonstrated for variations in $\mathbf{b}(x,y)$ and
$c(x,y)$, using an upwind lumped bilinear discontinuous (LBLD) finite element
discretization~\cite{Morel:2007}. A discontinuous discretization is used to
allow for discontinuities in material coefficient $c(x,y)$, and to account for the
fronts that develop in hyperbolic-type PDEs.
The transport problem,~\eqref{eq:transport}, results
in a highly non-symmetric (nearly lower-triangular in the proper ordering)
matrix $A$, for which traditional relaxation schemes such as Jacobi and
Gauss-Seidel diverge, even for small problems. Thus two sweeps of Gauss-Seidel
on the normal equations are chosen as the relaxation scheme, which contributes to a larger CC\@.

One anomalous feature of the LBLD discretization applied to~\eqref{eq:transport}
is that the best convergence rates are obtained with aggressive filtering,
$\theta_{pre} = \theta_{post} = 0.4-0.5$.  The rationale is that a large degree
sparsity pattern with aggressive filtering results in a sparsity pattern that is
long and narrow, following the direction of flow. For a hyperbolic PDE with the
solution following characteristic curves, effective sparsity
patterns for a given column of $P$ will align with the characteristic associated with the root-node.
However, in practice this requires either \textit{a
priori} knowledge of the problem to motivate aggressive filtering, or
experimentation. For most problems, filtering values of
$\theta_{pre}=\theta_{post} \in [0.05,0.15]$ are effective choices, increasing
to $\theta_{pre}=\theta_{post} \in [0.1,0.25]$ for problems with high
connectivity~---~e.g.\ 3d-diffusion as discussed in
Section~\ref{sec:numerical:3DAni}. However, the LBLD results demonstrate 
that further work is needed to algebraically determine the optimal interpolation sparsity
pattern for a given problem.

\begin{figure}[t!]
\begin{center}
\includegraphics[width=\textwidth]{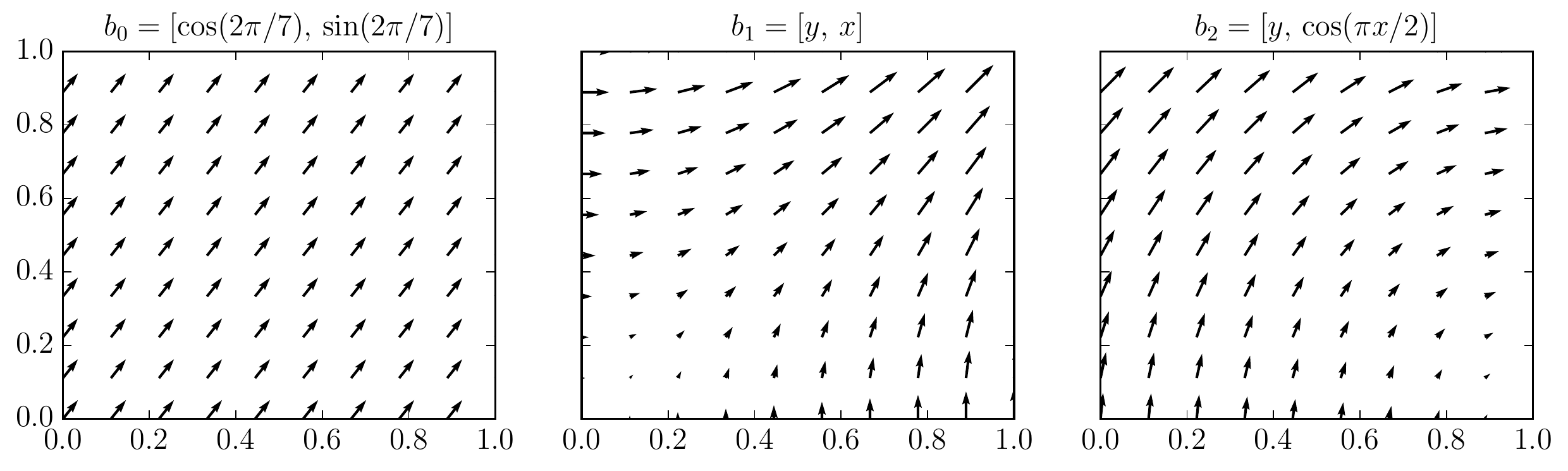}
\caption{Vector fields of flow functions, $\mathbf{b}(x,y)$ considered in~\eqref{eq:transport}, with inflow
boundaries $\Gamma_{in} =  \{ (x,y) \,:\, x = 0 \text{ or } y = 0\}$.}
\label{fig:flow}
\end{center}
\end{figure}

Root-node AMG is applied to the LBLD discretization of~\eqref{eq:transport} for three different material
coefficients $c(x,y)$, as shown in~\eqref{eq:material}, and three directional functions as shown in Figure~\ref{fig:flow}.
\begin{equation}\label{eq:material}
\begin{split}
\textnormal{Constant:} \hspace{3ex} c_1(x,y) & = 1 \\
\textnormal{Square in square (SnS):} \hspace{3ex} c_2(x,y) & = \begin{cases} 10^4 & x,y\in[0.25,0.75] \\ 10^{-4} & x,y\not\in[0.25,0.75]\end{cases} \\
\textnormal{Split:} \hspace{3ex} c_3(x,y) & = \begin{cases} 10^{-4} & x < 0.5\\ 10^{4} & x \geq 0.5 \end{cases}
\end{split}
\end{equation}
Results are shown in Table~\ref{tab:transport} for a discretization with 4,000,000 DOFs,
sufficiently large to observe asymptotic convergence factors.
A classical SOC with drop tolerance $\theta = 0.35$ is used, and GMRES energy-minimization applied
to a degree four sparsity pattern for interpolation and restriction operators in \rnamg, with filtering
$\theta_{pre}=\theta_{post} = 0.45$. The multigrid solver is used as a preconditioner
for GMRES, and problems solved to $10^{-8}$ residual tolerance. Candidate vectors are taken as the
constant and are \textit{not} improved for this problem, as convergence tended to degrade.
{\renewcommand{\tabcolsep}{6pt}

\begin{table}[!ht]
  \centering
  \begin{tabular}{c c | ccc | ccc | ccc }   \toprule
 & $\mathbf{b}(x,y)$ & \multicolumn{3}{c |}{$[\cos(\sfrac{2\pi}{7}),\sin(\sfrac{2\pi}{7})]$} & \multicolumn{3}{c |}{$[y,x]$} & \multicolumn{3}{c }{$[y,\cos(\sfrac{\pi x}{2})]$} \\
 & $c(x,y)$   & 1 & SnS & Split & 1 & SnS & Split & 1 & SnS & Split    \\ \midrule
\multirow{4}{*}{RN} & SC & 164 & 147 & 130 & 125 & 105 & 107 & 126 & 112 & 109 \\
			     & OC & 1.88 & 1.76 & 1.66 & 1.92 & 1.83 & 1.78 & 1.91 & 1.83 & 1.77 \\
			     & CC & 17.9 & 16.7 & 15.8 & 18.2 & 17.4 & 17.0 & 18.2 & 17.4 & 16.8 \\
			     & $\rho$ & 0.79 & 0.74 & 0.75 & 0.86 & 0.84 & 0.85 & 0.87 & 0.87 & 0.86 \\ \midrule			     
\end{tabular}
\caption{Non-symmetric \rnamg\ for variations on the 2d steady-state transport
equation~\eqref{eq:transport}. Matrices have $4,000,000$ DOFs and $15,986,004$ nonzeros. }
\label{tab:transport}
\end{table}
}

Although convergence factors in the case of non-constant flow are higher than desired (worst case $\rho\approx0.87$
for \rnamg), it is encouraging that AMG methods are able to solve upwind discretizations of a
hyperbolic PDE with discontinuous and non-constant coefficient functions. Convergence factors
of \saamg quickly approached one, and did not converge in 500 iterations. 

Often when considering transport-type problems, ``sweeps'' are performed, where the problem is
discretized in angle and a linear solve performed for each angle. Transport sweeps are an important
part of the DSA algorithm for models of neutral particle transport~\cite{adams2002fast}. For this reason, it is of
interest to demonstrate \rnamg's capability on the entire spectrum of angles, similar to anisotropic
diffusion as considered in Section~\ref{sec:numerical:TAD}. Figure~\ref{fig:ang_lbld} shows convergence
factors of \rnamg\ as a function of direction of flow, $\theta\in[0,\sfrac{\pi}{2}]$, as applied to the LBLD
discretization of~\eqref{eq:transport}.
\begin{figure}[!t]
\begin{center}
\includegraphics[width=0.95\textwidth]{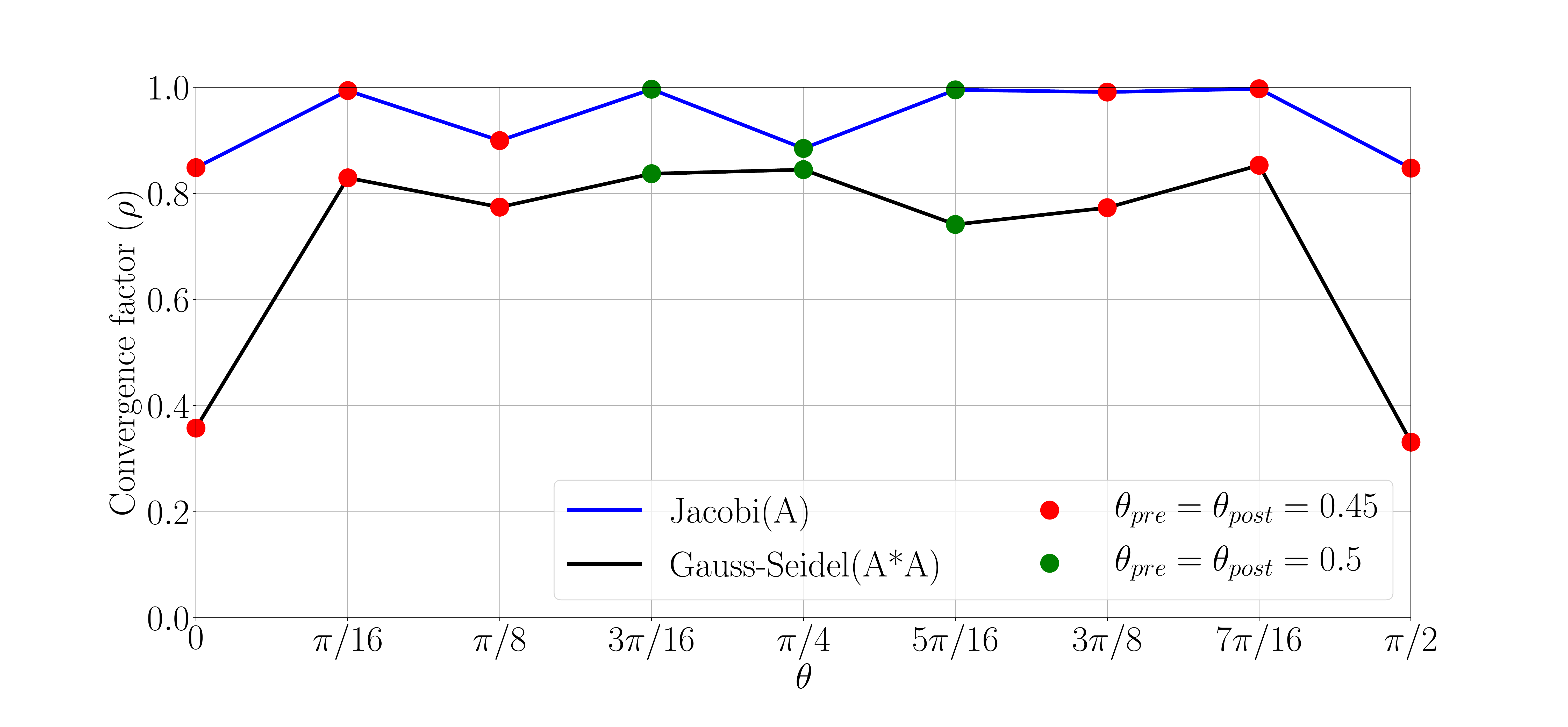}
\caption{Convergence factors for GMRES-accelerated \rnamg\ applied to the steady-state transport
equation, as a function of the direction of flow $\mathbf{b}(x,y) = (\cos(\theta),\sin(\theta))$.
Results are shown for Gauss-Seidel on the normal
equations (GSNE) relaxation as well as Jacobi, to demonstrate why normal-equation relaxation is necessary
for highly non-symmetric problems. Other solver parameters are fixed for all tests, with the exception
of filtering, which is modified slightly for improved convergence on angles close to $\sfrac{\pi}{4}$.
Convergence with GSNE is in the range $[0.75,0.85]$ for interior angles, with significant improvement
close to the boundaries.}\label{fig:ang_lbld}
\end{center}
\end{figure}

\subsection{Elasticity example (systems problem)}\label{sec:elasticity}

Algebraic multigrid for systems of PDEs is an important topic
due to the vast number of problems formulated this way.
Although there are some problem-specific extensions to CF AMG
for systems, such as elasticity~\cite{BaKoYa2010}, where rigid body modes are
incorporated into a modified method, CF AMG is not applicable to
systems in a general setting. In contrast, \saamg\ specifically targets
systems-based problems such as elasticity, and performs well in many cases.
Here, the suitability of \rnamg\ for systems is illustrated through a
comparison with \saamg\ for a standard elasticity model problem, a problem for
which \saamg\ was designed.

The test problem is isotropic linearized elasticity defined by
\begin{equation}\label{eqn:linelas}
  -\mathrm{div}\left( \lambda \,\mathrm{tr}\left( \left( \nabla \mathbf{u}
  + \nabla \mathbf{u}^T \right)/2 \right) I + \mu \left( \nabla
  \mathbf{u} + \nabla \mathbf{u}^T \right) \right) = f,
\end{equation}
where $\lambda$ and $\mu$ are the Lam\'{e} parameters, $I$ is the identity
matrix, and $\mbox{tr}()$ is the trace function.  Example two from the MFEM
package~\cite{mfem} is used to discretize (\ref{eqn:linelas}) using a regular
mesh of a beam that is eight times longer than it is wide (\texttt{data/beam-tri.mesh}).
Three material choices, of increasing difficulty for AMG, are considered:
a steel beam, with Young's modulus and Poisson ratio of $E = 180\mbox{E}9$
and $\nu = 0.30$; a rubber beam, with $E = 0.1\mbox{E}9$ and $\nu = 0.499$; and a
more difficult rubber beam, where $\nu = 0.4999$.  In general, the closer $\nu$ is to
0.5, the more difficult the system is for AMG to solve.
The east side of the beam is fixed to a wall with $\mathbf{u} =
\mathbf{0}$ Dirichlet boundary conditions. Neumann conditions are applied on other boundaries,
with $\mathbf{u} \cdot \mathbf{n} = f$, where $f = 0$ for the north and
south sides, and $f = -0.01$ on the west side, representing a downward force.
Standard linear triangular finite elements are used to produce an SPD
system with two variables at each spatial location and the corresponding block
sparse row matrix with $2 \times 2$ blocks.

Each example uses a V$(1,1)$-cycle of symmetric block Gauss-Seidel with CG
acceleration. Both \rnamg\ and \saamg\ use the classical strength measure with a
drop tolerance of $0.5$.  Root-node AMG does not use the evolution measure here
because the classic measure performed well and lowered the setup complexity.\footnote{If the three elasticity test cases are re-run for $\textnormal{RN}_4$ using the
evolution strength measure, then the convergence rate degrades by about 0.01--0.02 and
the operator complexities remain similar,
i.e., \rnamg\ using the evolution measure converges similarly to SA using the
classic measure. While SOC for systems is an active topic of AMG research and not well-understood,
one possible reason for this small difference between the two measures is that
the ``anisotropy'' in the beam is grid-aligned, and the classical measure is known
to handle grid-aligned anisotropies well.}
The terms $\textnormal{SA}_{k}$
and $\textnormal{RN}_{k}$ again refer to the degree of interpolation smoothing.
For $\textnormal{RN}_{4}$, 6 iterations of CG energy minimization are used, while for
$\textnormal{RN}_{1}$, 2 iterations are used.  Experimentation indicates that only
pre-filtering is needed, so $\theta$ values only refer to pre-filtering.  The candidate
vectors $B_0$ are the three rigid-body-modes in 2D, two translations and one rotation.

Table~\ref{tab:elas} gives the numerical results for three problem sizes,
from 256K to 4M DOFs.  Various filtering
values produce qualitatively similar results, so a single result is reported.
All methods converge well for the steel beam,  with the 0.5 convergence rate for
$\textnormal{RN}_{1}$ being offset by the low operator complexity of 1.2. The rubber
beams are considerably harder, with the most difficult beam problem yielding convergence
rates well above 0.9.

Comparing the methods with each other, the results reveal that
$\textnormal{RN}_{1}$ converges slowly in comparison to $\textnormal{SA}_{1}$
or $\textnormal{RN}_{4}$.  This is because the three candidate vectors provided
as interpolation constraints require a richer sparsity pattern than $d=1$ in
order to satisfy the constraints.

Another important point is that \rnamg\ with $d=4$ does achieve a better convergence rate
and smaller operator complexity than \saamg\, but at the cost of a larger SC\@.
Taking the largest problem size for test 2, with
4M DOFs (a grid size $4096\times512$), the difference in convergence
rate yields 87 iterations for $\textnormal{RN}_{4}$ versus 107 iterations for
$\textnormal{SA}_{1}$.  However even with this faster convergence, \saamg\ uses
the fewest overall work units.\footnote{It is possible, given the faster creep
in convergence rate suffered by \saamg, that \rnamg\ could outperform
overall as the problem size increases.} \saamg\ similarly provides slightly better
overall work units for setup and solve with test 3. In conclusion, this section shows
that \rnamg\ is a viable approach to solving systems, and is competitive with
\saamg\ for the considered elasticity problem, for which SA AMG was originally designed.
Applying \rnamg\ to systems is
still an active research topic, focusing on issues such as reducing the setup
cost and considering more complicated examples.
{\renewcommand{\tabcolsep}{4.3pt}
\begin{table}[!ht]
  \centering
  \begin{tabular}{c c c | cccc | cccc | cccc }   \toprule
  &      \multicolumn{2}{r}{Grid size:}      & \multicolumn{4}{c |}{$1024\times128$} & \multicolumn{4}{c |}{$2048\times256$} & \multicolumn{4}{c }{$4096\times512$}\\
Test &      $\theta$      & $d$                                & SC   &   OC &   CC  & $\rho$         & SC   &   OC &   CC  & $\rho$         &  SC  &  OC  &  CC  &  $\rho$  \\
     \midrule
 \multirow{3}{*}{1} &\multirow{1}{*}{--} & $\textnormal{SA}_1$ & 105  &  1.6 & 9.8   & 0.23       &  106 & 1.7  &  9.8  & 0.26       &  105 & 1.7  & 9.8  & 0.26 \\
                    &\multirow{2}{*}{0.1}& $\textnormal{RN}_1$ & 84   &  1.2 & 6.9   & 0.50       &  84  & 1.2  &  6.9  & 0.49       &  85  & 1.2  & 6.9  & 0.51 \\
                    &                    & $\textnormal{RN}_4$ & 226  &  1.4 & 8.2   & 0.19       & 226  & 1.4  &  8.2  & 0.21       &  225 & 1.4  & 8.2  & 0.19 \\
     \midrule
 \multirow{3}{*}{2} &\multirow{1}{*}{--} & $\textnormal{SA}_1$ & 104  & 1.6  & 9.6   & 0.83       & 106  & 1.6  & 9.8   & 0.84       &  104 &  1.7 &  9.8 & 0.84 \\
                    &\multirow{2}{*}{0.1}& $\textnormal{RN}_1$ & 77   & 1.2  & 6.7   & 0.92       & 77   & 1.2  & 6.7   & 0.94       &  78  &  1.2 &  6.7 & 0.93 \\
                    &                    & $\textnormal{RN}_4$ & 221  & 1.4  & 8.2   & 0.80       & 223  & 1.4  & 8.2   & 0.80       &  223 &  1.4 &  8.3 & 0.81 \\
     \midrule
 \multirow{3}{*}{3} &\multirow{1}{*}{--} & $\textnormal{SA}_1$ & 104  & 1.6  &  9.7  & 0.93       & 105  & 1.6  & 9.8   & 0.94       &  105 &  1.7 &  9.8 & 0.93 \\
                    &\multirow{2}{*}{0.1}& $\textnormal{RN}_1$ & 82   & 1.2  &  6.7  & 0.97       &  82  & 1.2  & 6.7   & 0.97       &   81 &  1.2 &  6.7 & 0.97 \\
                    &                    & $\textnormal{RN}_4$ & 228  & 1.4  &  8.2  & 0.92       & 229  & 1.4  & 8.3   & 0.92       &  229 &  1.4 &  8.3 & 0.92 \\
  \bottomrule
  \end{tabular}

  \caption{Application of \saamg\ and \rnamg\ to three linearized elasticity
   problems.  The tests are progressively more difficult, with Test 1 being a
   steel bar with Poisson ratio $\nu=0.30$, Test 2 being a rubber bar with
   $\nu = 0.499$, and Test 3 being a rubber bar with $\nu=0.4999$.}\label{tab:elas}
\end{table}
}

\begin{remark}

   It is important to note that the ideal $P$ in equation (\ref{eq:ideal}) is
   only ideal for the specific two-grid setting outlined in~\cite{FaVa2004}.
   This ideal $P$ is not guaranteed to be ideal for multilevel methods, or even
   for all two-grid methods.  However, the ideal $P$ is still a useful target
   for interpolation as evidenced by the experiments in this work and by the
   fact this it is ideal under the assumptions in~\cite{FaVa2004}.
   Along with approximating $P_{\textnormal{ideal}}$, \rnamg
   combines other proven approaches to constructing quality
   interpolation operators, interpolation smoothing and guaranteed
   interpolation of known algebraically smooth modes.

   As an example of how
   this combination of strategies makes \rnamg\ more robust than using
  $P_{\textnormal{ideal}}$, two-grid experiments for the three test cases from
  Table~\ref{tab:elas} are run on a small mesh yielding 4.5K DOFs.  For the steel
   beam (Test 1), the convergence factors are 0.19, 0.17 and 0.18, for
   $\textnormal{SA}_1$, $\textnormal{RN}_4$ and ideal $P$, respectively.
   Moving to the rubber beam (Test 2), the convergence factors are 0.80, 0.79
   and 0.97.  Finally for the harder rubber beam (Test 3), the convergence
   factors are 0.92, 0.91, 0.98.  Experimentally, it appears that as the
   Poisson ratio approaches 0.5, convergence $P_{\textnormal{ideal}}$ degrades significantly faster
   than the interpolation operators produced by \saamg\ and \rnamg.

\end{remark}

\section{Conclusions}\label{sec:conclusions}

The presented RN AMG methodology has proven successful on a wide variety of
problems. The proposed filtering of the interpolation sparsity patterns before and after
energy-minimization smoothing iterations greatly reduces setup and cycle complexity,
and at times improves convergence, making RN AMG a robust solver for difficult SPD
problems, non-symmetric problems, and systems-related problems. Total complexity
remains reasonable for all of the examples tested, and tends to scale with the size of the problem.
Providing such total complexity estimates, including the setup phase, is a specific contribution
of this work.

One particularly difficult problem is strongly anisotropic diffusion, which
RN AMG is able to solve effectively in comparison to other AMG methods. Because the focus
here is on a general root-node methodology, coupled with energy-minimization smoothing
of $P$, minimal testing is done in coupling RN AMG with other advancements in AMG\@.
In the case of strongly anisotropic diffusion, many other works consider modified SOC and
advanced
coarsening schemes in order to better capture the anisotropy~\cite{OlScTu2009,DAmbra:2013iwa,
BrBrKaLi2015,Notay:2010um,Brannick:2012bl,BrBrMaMaMcRu2006,BrBrKaLi2015}, thereby
improving convergence rates. Future work involves coupling RN AMG with a
larger variety of SOC measures, aggregation routines, and sparsity patterns for $P$,
to further improve convergence on anisotropic problems. Root-node is also shown effective at
solving convection-diffusion, and a discontinuous, upwind discretization of a hyperbolic PDE
(steady-state transport), demonstrating that AMG need not be limited to elliptic problems.

Parallelization of RN AMG is straightforward. The basic computational kernels
are typically available in AMG codes (e.g., matrix-matrix multiply),
are light-weight (e.g., filtering entries), or purely local in computation (e.g., the
row-wise projection operation (\ref{eqn:mode_constraint})).

New theoretical motivation for coupling energy minimization with RN AMG is provided in
the symmetric and non-symmetric setting. Proving Conjecture~\ref{conj_stab} would complete
the two-grid convergence proof for ideal operators in the non-symmetric case. However,
the stability constraint is of limited use in practice, as it is not directly approximated like the
weak and strong approximation properties. Further work on developing energy-based,
non-symmetric convergence theory is important and ongoing work in understanding how
to construct AMG solvers for non-symmetric systems.

\bibliographystyle{siamplain}
\bibliography{amg.bib}

\begin{thebibliography}{10}

\bibitem{mfem}
{\em {MFEM}: Modular finite element methods}.
\newblock \url{mfem.org}.

\bibitem{adams2002fast}
{\sc M.~L. Adams and E.~W. Larsen}, {\em Fast iterative methods for
  discrete-ordinates particle transport calculations}, Progress in nuclear
  energy, 40 (2002), pp.~3--159.

\bibitem{Baker2012}
{\sc A.~H. Baker, R.~D. Falgout, T.~Gamblin, T.~V. Kolev, M.~Schulz, and U.~M.
  Yang}, {\em Competence in High Performance Computing 2010: Proceedings of an
  International Conference on Competence in High Performance Computing, June
  2010, Schloss Schwetzingen, Germany}, Springer Berlin Heidelberg, Berlin,
  Heidelberg, 2012, ch.~Scaling Algebraic Multigrid Solvers: On the Road to
  Exascale, pp.~215--226,
  \href{http://dx.doi.org/10.1007/978-3-642-24025-6_18}{doi:\nolinkurl{10.1007/978-3-642-24025-6_18}}.

\bibitem{BaFaKoYa2011}
{\sc A.~H. Baker, R.~D. Falgout, T.~V. Kolev, and U.~M. Yang}, {\em Multigrid
  smoothers for ultraparallel computing}, SIAM J. Sci. Comput., 33 (2011),
  pp.~2864--2887,
  \href{http://dx.doi.org/10.1137/100798806}{doi:\nolinkurl{10.1137/100798806}}.

\bibitem{BaKoYa2010}
{\sc A.~H. Baker, T.~V. Kolev, and U.~M. Yang}, {\em Improving algebraic
  multigrid interpolation operators for linear elasticity problems}, Numerical
  Linear Algebra with Applications, 17 (2010), pp.~495--517,
  \href{http://dx.doi.org/10.1002/nla.688}{doi:\nolinkurl{10.1002/nla.688}}.

\bibitem{BeOlSc2008}
{\sc W.~N. Bell, L.~N. Olson, and J.~B. Schroder}, {\em {PyAMG}: Algebraic
  multigrid solvers in {Python} v3.0}, 2015, \url{http://www.pyamg.org}.
\newblock Release 3.0.

\bibitem{Brandt:2000vn}
{\sc A.~Brandt}, {\em {General highly accurate algebraic coarsening}},
  Electronic Transactions on Numerical Analysis, 10 (2000), pp.~1--20.

\bibitem{BrBrKaLi2011}
{\sc A.~Brandt, J.~Brannick, K.~Kahl, and I.~Livshits}, {\em Bootstrap amg},
  SIAM Journal on Scientific Computing, 33 (2011), pp.~612--632,
  \href{http://dx.doi.org/10.1137/090752973}{doi:\nolinkurl{10.1137/090752973}}.

\bibitem{BrBrKaLi2015}
{\sc A.~Brandt, J.~Brannick, K.~Kahl, and I.~Livshits}, {\em Algebraic distance
  for anisotropic diffusion problems: multilevel results}, Electronic
  Transactions on Numerical Analysis, 44 (2015), pp.~472--496.

\bibitem{Brandt:2011tu}
{\sc A.~Brandt, J.~J. Brannick, K.~Kahl, and I.~Livshits}, {\em {An algebraic
  distances measure of AMG strength of connection}}, June 2011.

\bibitem{BrMcRu1984}
{\sc A.~Brandt, S.~F. McCormick, and J.~W. Ruge}, {\em Algebraic multigrid
  ({AMG}) for sparse matrix equations}, in Sparsity and Its Applications, D.~J.
  Evans, ed., Cambridge Univ. Press, Cambridge, 1984, pp.~257--284.

\bibitem{BrBrMaMaMcRu2006}
{\sc J.~Brannick, M.~Brezina, S.~MacLachlan, T.~Manteuffel, S.~McCormick, and
  J.~Ruge}, {\em An energy-based amg coarsening strategy}, Numerical Linear
  Algebra with Applications, 13 (2006), pp.~133--148,
  \href{http://dx.doi.org/10.1002/nla.480}{doi:\nolinkurl{10.1002/nla.480}}.

\bibitem{Brannick:2012bl}
{\sc J.~Brannick, Y.~Chen, and L.~Zikatanov}, {\em An algebraic multilevel
  method for anisotropic elliptic equations based on subgraph matching},
  Numerical Linear Algebra with Applications, 19 (2012), pp.~279--295,
  \href{http://dx.doi.org/10.1002/nla.1804}{doi:\nolinkurl{10.1002/nla.1804}}.

\bibitem{Brannick:2007fb}
{\sc J.~Brannick and L.~Zikatanov}, {\em Algebraic Multigrid Methods Based on
  Compatible Relaxation and Energy Minimization}, Springer Berlin Heidelberg,
  Berlin, Heidelberg, 2007, pp.~15--26,
  \href{http://dx.doi.org/10.1007/978-3-540-34469-8_2}{doi:\nolinkurl{10.1007/978-3-540-34469-8_2}}.

\bibitem{Brannick:2010hz}
{\sc J.~J. Brannick and R.~D. Falgout}, {\em Compatible relaxation and
  coarsening in algebraic multigrid}, SIAM Journal on Scientific Computing, 32
  (2010), pp.~1393--1416,
  \href{http://dx.doi.org/10.1137/090772216}{doi:\nolinkurl{10.1137/090772216}}.

\bibitem{BrFaMaMaMcRu2005}
{\sc M.~Brezina, R.~Falgout, S.~MacLachlan, T.~Manteuffel, S.~McCormick, and
  J.~Ruge}, {\em Adaptive smoothed aggregation ($\alpha${SA}) multigrid}, SIAM
  Review, 47 (2005), pp.~317--346,
  \href{http://dx.doi.org/10.1137/050626272}{doi:\nolinkurl{10.1137/050626272}}.

\bibitem{Brezina:2006bt}
{\sc M.~Brezina, R.~Falgout, S.~MacLachlan, T.~Manteuffel, S.~McCormick, and
  J.~Ruge}, {\em Adaptive algebraic multigrid}, SIAM Journal on Scientific
  Computing, 27 (2006), pp.~1261--1286,
  \href{http://dx.doi.org/10.1137/040614402}{doi:\nolinkurl{10.1137/040614402}}.

\bibitem{BrMaMcRuSa2010}
{\sc M.~Brezina, T.~Manteuffel, S.~MCormick, J.~Ruge, and G.~Sanders}, {\em
  Towards adaptive smoothed aggregation ($\alpha${SA}) for nonsymmetric
  problems}, SIAM Journal on Scientific Computing, 32 (2010), pp.~14--39,
  \href{http://dx.doi.org/10.1137/080727336}{doi:\nolinkurl{10.1137/080727336}}.

\bibitem{Brezina:2012kl}
{\sc M.~Brezina, P.~Vaněk, and P.~S. Vassilevski}, {\em An improved
  convergence analysis of smoothed aggregation algebraic multigrid}, Numerical
  Linear Algebra with Applications, 19 (2012), pp.~441--469,
  \href{http://dx.doi.org/10.1002/nla.775}{doi:\nolinkurl{10.1002/nla.775}}.

\bibitem{Chan:2000tl}
{\sc T.~F. Chan and P.~Vanek}, {\em Detection of Strong Coupling in Algebraic
  Multigrid Solvers}, Springer Berlin Heidelberg, Berlin, Heidelberg, 2000,
  pp.~11--23,
  \href{http://dx.doi.org/10.1007/978-3-642-58312-4_2}{doi:\nolinkurl{10.1007/978-3-642-58312-4_2}}.

\bibitem{Chen:2015vx}
{\sc M.-H. Chen and A.~Greenbaum}, {\em Analysis of an aggregation-based
  algebraic two-grid method for a rotated anisotropic diffusion problem},
  Numerical Linear Algebra with Applications, 22 (2015), pp.~681--701,
  \href{http://dx.doi.org/10.1002/nla.1980}{doi:\nolinkurl{10.1002/nla.1980}}.
\newblock nla.1980.

\bibitem{DAmbra:2013iwa}
{\sc P.~D'Ambra and P.~S. Vassilevski}, {\em Adaptive amg with coarsening based
  on compatible weighted matching}, Computing and Visualization in Science, 16
  (2013), pp.~59--76,
  \href{http://dx.doi.org/10.1007/s00791-014-0224-9}{doi:\nolinkurl{10.1007/s00791-014-0224-9}}.

\bibitem{DeFaNoYa_2008}
{\sc H.~De~Sterck, R.~D. Falgout, J.~W. Nolting, and U.~M. Yang}, {\em
  Distance-two interpolation for parallel algebraic multigrid}, Numerical
  Linear Algebra with Applications, 15 (2008), pp.~115--139,
  \href{http://dx.doi.org/10.1002/nla.559}{doi:\nolinkurl{10.1002/nla.559}}.

\bibitem{ElSiWa2014}
{\sc H.~C. Elman, D.~J. Silvester, and A.~J. Wathen}, {\em Finite elements and
  fast iterative solvers: with applications in incompressible fluid dynamics},
  Oxford University Press (UK), 2014.

\bibitem{FaVa2004}
{\sc R.~D. Falgout and P.~S. Vassilevski}, {\em On generalizing the algebraic
  multigrid framework}, SIAM Journal on Numerical Analysis, 42 (2004),
  pp.~1669--1693,
  \href{http://dx.doi.org/10.1137/S0036142903429742}{doi:\nolinkurl{10.1137/S0036142903429742}}.

\bibitem{hypre}
{\sc R.~D. Falgout and U.~M. Yang}, {\em hypre: A Library of High Performance
  Preconditioners}, Springer Berlin Heidelberg, Berlin, Heidelberg, 2002,
  pp.~632--641,
  \href{http://dx.doi.org/10.1007/3-540-47789-6_66}{doi:\nolinkurl{10.1007/3-540-47789-6_66}}.

\bibitem{Gee:2009dy}
{\sc M.~W. Gee, J.~J. Hu, and R.~S. Tuminaro}, {\em A new smoothed aggregation
  multigrid method for anisotropic problems}, Numerical Linear Algebra with
  Applications, 16 (2009), pp.~19--37,
  \href{http://dx.doi.org/10.1002/nla.593}{doi:\nolinkurl{10.1002/nla.593}}.

\bibitem{BoomerAMG}
{\sc V.~E. Henson and U.~M. Yang}, {\em {BoomerAMG}: A parallel algebraic
  multigrid solver and preconditioner}, Appl. Numer. Math., 41 (2002),
  pp.~155--177,
  \href{http://dx.doi.org/10.1016/S0168-9274(01)00115-5}{doi:\nolinkurl{10.1016/S0168-9274(01)00115-5}}.

\bibitem{hu:2016}
{\sc X.~Hu, P.~S. Vassilevski, and J.~Xu}, {\em A two-grid sa-amg convergence
  bound that improves when increasing the polynomial degree}, Numerical Linear
  Algebra with Applications, 23 (2016), pp.~746--771.

\bibitem{KeMaSc2015}
{\sc C.~Ketelsen, T.~Manteuffel, and J.~B. Schroder}, {\em Least-squares finite
  element discretization of the neutron transport equation in spherical
  geometry}, SIAM Journal on Scientific Computing, 37 (2015), pp.~S71--S89,
  \href{http://dx.doi.org/10.1137/140975152}{doi:\nolinkurl{10.1137/140975152}}.

\bibitem{Livne:2004vt}
{\sc O.~E. Livne}, {\em Coarsening by compatible relaxation}, Numerical Linear
  Algebra with Applications, 11 (2004), pp.~205--227,
  \href{http://dx.doi.org/10.1002/nla.378}{doi:\nolinkurl{10.1002/nla.378}}.

\bibitem{2014MaOlamgtheory}
{\sc S.~P. MacLachlan and L.~N. Olson}, {\em Theoretical bounds for algebraic
  multigrid performance: review and analysis}, Numerical Linear Algebra with
  Applications, 21 (2014), pp.~194--220,
  \href{http://dx.doi.org/10.1002/nla.1930}{doi:\nolinkurl{10.1002/nla.1930}}.

\bibitem{MaBrVa1999}
{\sc J.~Mandel, M.~Brezina, and P.~Van{\v{e}}k}, {\em Energy optimization of
  algebraic multigrid bases}, Computing, 62 (1999), pp.~205--228,
  \href{http://dx.doi.org/10.1007/s006070050022}{doi:\nolinkurl{10.1007/s006070050022}}.

\bibitem{Morel:2007}
{\sc J.~E. Morel and J.~S. Warsa}, {\em Spatial finite-element lumping
  techniques for the quadrilateral mesh s n equations in x-y geometry}, Nuclear
  science and engineering, 156 (2007), pp.~325--342.

\bibitem{Napov:2011tj}
{\sc A.~Napov and Y.~Notay}, {\em Algebraic analysis of aggregation-based
  multigrid}, Numerical Linear Algebra with Applications, 18 (2011),
  pp.~539--564,
  \href{http://dx.doi.org/10.1002/nla.741}{doi:\nolinkurl{10.1002/nla.741}}.

\bibitem{Notay:2010um}
{\sc Y.~Notay}, {\em An aggregation-based algebraic multigrid method.}, ETNA.
  Electronic Transactions on Numerical Analysis [electronic only], 37 (2010),
  pp.~123--146.

\bibitem{Notay:2010em}
{\sc Y.~Notay}, {\em Algebraic analysis of two-grid methods: The nonsymmetric
  case}, Numerical Linear Algebra with Applications, 17 (2010), pp.~73--96,
  \href{http://dx.doi.org/10.1002/nla.649}{doi:\nolinkurl{10.1002/nla.649}}.

\bibitem{Notay:2014uc}
{\sc Y.~Notay}, {\em Algebraic theory of two-grid methods}, Numerical
  Mathematics: Theory, Methods and Applications, 8 (2015), pp.~168--198,
  \href{http://dx.doi.org/10.4208/nmtma.2015.w04si}{doi:\nolinkurl{10.4208/nmtma.2015.w04si}}.

\bibitem{OlScTu2009}
{\sc L.~N. Olson, J.~Schroder, and R.~S. Tuminaro}, {\em A new perspective on
  strength measures in algebraic multigrid}, Numerical Linear Algebra with
  Applications, 17 (2010), pp.~713--733,
  \href{http://dx.doi.org/10.1002/nla.669}{doi:\nolinkurl{10.1002/nla.669}}.

\bibitem{2011_OlSc_scidac11}
{\sc L.~N. Olson and J.~B. Schroder}, {\em Components of a more robust
  multilevel solver for emerging architectures and complex applications}, in
  SciDAC 2011, Denver, CO, July 10-14 2011.

\bibitem{OlScTu2011}
{\sc L.~N. Olson, J.~B. Schroder, and R.~S. Tuminaro}, {\em A general
  interpolation strategy for algebraic multigrid using energy minimization},
  SIAM Journal on Scientific Computing, 33 (2011), pp.~966--991,
  \href{http://dx.doi.org/10.1137/100803031}{doi:\nolinkurl{10.1137/100803031}}.

\bibitem{RuStu1987}
{\sc J.~W. Ruge and K.~St{\"u}ben}, {\em Algebraic multigrid}, in Multigrid
  methods, vol.~3 of Frontiers in Applied Mathematics, SIAM, Philadelphia, PA,
  1987, pp.~73--130.

\bibitem{Sala:2008cv}
{\sc M.~Sala and R.~S. Tuminaro}, {\em A new petrov–galerkin smoothed
  aggregation preconditioner for nonsymmetric linear systems}, SIAM Journal on
  Scientific Computing, 31 (2008), pp.~143--166,
  \href{http://dx.doi.org/10.1137/060659545}{doi:\nolinkurl{10.1137/060659545}}.

\bibitem{Schaffer:1998ui}
{\sc S.~Schaffer}, {\em A semicoarsening multigrid method for elliptic partial
  differential equations with highly discontinuous and anisotropic
  coefficients}, SIAM Journal on Scientific Computing, 20 (1998), pp.~228--242,
  \href{http://dx.doi.org/10.1137/S1064827595281587}{doi:\nolinkurl{10.1137/S1064827595281587}}.

\bibitem{Sc2012}
{\sc J.~B. Schroder}, {\em Smoothed aggregation solvers for anisotropic
  diffusion}, Numerical Linear Algebra with Applications, 19 (2012),
  pp.~296--312,
  \href{http://dx.doi.org/10.1002/nla.1805}{doi:\nolinkurl{10.1002/nla.1805}}.

\bibitem{Van:2001bw}
{\sc P.~Van{\v{e}}k, M.~Brezina, and J.~Mandel}, {\em Convergence of algebraic
  multigrid based on smoothed aggregation}, Numerische Mathematik, 88 (2001),
  pp.~559--579,
  \href{http://dx.doi.org/10.1007/s211-001-8015-y}{doi:\nolinkurl{10.1007/s211-001-8015-y}}.

\bibitem{Janka:1999tc}
{\sc P.~Vanek, A.~Janka, and H.~Guillard}, {\em {Convergence of Algebraic
  Multigrid Based on Smoothed Aggregation II: Extension to a Petrov-Galerkin
  Method}}, Tech. Report RR-3683, {INRIA}, May 1999,
  \url{https://hal.inria.fr/inria-00072986}.

\bibitem{VaMaBr1996}
{\sc P.~Van{\v{e}}k, J.~Mandel, and M.~Brezina}, {\em Algebraic multigrid by
  smoothed aggregation for second and fourth order elliptic problems},
  Computing, 56 (1996), pp.~179--196.

\bibitem{Va2008}
{\sc P.~S. Vassilevski}, {\em Multilevel block factorization preconditioners:
  Matrix-based Analysis and Algorithms for Solving Finite Element Equations},
  Springer New York, 2008,
  \href{http://dx.doi.org/10.1007/978-0-387-71564-3}{doi:\nolinkurl{10.1007/978-0-387-71564-3}}.

\bibitem{Wiesner:2014cy}
{\sc T.~A. Wiesner, R.~S. Tuminaro, W.~A. Wall, and M.~W. Gee}, {\em Multigrid
  transfers for nonsymmetric systems based on schur complements and galerkin
  projections}, Numerical Linear Algebra with Applications, 21 (2014),
  pp.~415--438,
  \href{http://dx.doi.org/10.1002/nla.1889}{doi:\nolinkurl{10.1002/nla.1889}}.

\end{thebibliography}

\end{document}